\documentclass[11pt]{article}

\usepackage{epsfig,epsf,fancybox}
\usepackage{subfigure}
\usepackage[english]{babel}
\usepackage{amsmath}
\usepackage{mathrsfs}
\usepackage{amssymb}
\usepackage{graphicx}
\usepackage{color}
\usepackage{boxedminipage}
\usepackage{stmaryrd}
\usepackage{multirow}
\usepackage{booktabs}
\usepackage{accents}
\usepackage{cite}
\usepackage{float}
\usepackage[ruled,vlined,linesnumbered]{algorithm2e}
\usepackage[linktocpage,pagebackref,colorlinks,linkcolor=blue,anchorcolor=blue,citecolor=blue,urlcolor=blue]{hyperref}

\newtheorem{theorem}{Theorem}[section]

\newtheorem{lemma}[theorem]{Lemma}
\newtheorem{proposition}[theorem]{Proposition}

\,
\,

\newcommand{\be}{\begin{equation}}
\newcommand{\ee}{\end{equation}}
\newcommand{\ba}{\begin{array}}
\newcommand{\ea}{\end{array}}
\newcommand{\bpm}{\begin{pmatrix}}
\newcommand{\epm}{\end{pmatrix}}

\newcommand{\bea}{\begin{eqnarray}}
\newcommand{\eea}{\end{eqnarray}}
\newcommand{\beaa}{\begin{eqnarray*}}
\newcommand{\eeaa}{\end{eqnarray*}}

\newcommand{\bal}{\begin{align}}
\newcommand{\eal}{\end{align}}
\newcommand{\baln}{\begin{align*}}
\newcommand{\ealn}{\end{align*}}

\usepackage[normalem]{ulem} 

\newcommand{\Xcal}{\mathcal{X}}

\newcommand{\Ycal}{\mathcal{Y}}

\newcommand{\CL}{\mathcal L}

\newcommand{\vb}{b}

\newcommand{\vr}{r}

\newcommand{\vu}{u}
\newcommand{\vv}{v}
\newcommand{\vw}{w}
\newcommand{\vx}{x}
\newcommand{\vy}{y}
\newcommand{\vz}{z}

\newcommand{\vlam}{\lambda}
\newcommand{\vdelta}{\delta}

\newcommand{\vA}{A}
\newcommand{\vB}{B}

\newcommand{\vD}{D}

\newcommand{\vG}{G}

\newcommand{\vI}{I}

\newcommand{\vP}{P}
\newcommand{\vQ}{Q}

\newcommand{\vU}{U}

\newcommand{\vW}{W}

\newcommand{\cB}{{\mathcal{B}}}

\newcommand{\cS}{{\mathcal{S}}}

\newcommand{\cX}{{\mathcal{X}}}
\newcommand{\cY}{{\mathcal{Y}}}
\newcommand{\cZ}{{\mathcal{Z}}}

\newcommand{\EE}{\mathbb{E}} 
\newcommand{\RR}{\mathbb{R}} 
\newcommand{\vzero}{0} 

\newcommand{\Prob}{{\mathrm{Prob}}} 

\newcommand{\st}{\mbox{ s.t. }}

\DeclareMathOperator*{\argmin}{arg\,min} 

\newcommand{\bc}{\begin{center}}
\newcommand{\ec}{\end{center}}

\newcommand{\bdm}{\begin{displaymath}}
\newcommand{\edm}{\end{displaymath}}

\newcommand{\beq}{\begin{equation}}
\newcommand{\eeq}{\end{equation}}

\newcommand{\bfl}{\begin{flushleft}}
\newcommand{\efl}{\end{flushleft}}

\newcommand{\bt}{\begin{tabbing}}
\newcommand{\et}{\end{tabbing}}

\newcommand{\beqn}{\begin{eqnarray}}
\newcommand{\eeqn}{\end{eqnarray}}

\newcommand{\beqs}{\begin{align*}} 
\newcommand{\eeqs}{\end{align*}}  

\newtheorem{remark}{Remark}[section]
\newtheorem{assumption}{Assumption}

\begin{document}

\title{Randomized Primal-Dual Proximal Block Coordinate Updates}

\author{Xiang Gao\thanks{\{gaoxx460, zhangs\}@umn.edu. Department of Industrial \& Systems Engineering, University of Minnesota} \and
Yangyang Xu\thanks{yangyang.xu@ua.edu. Department of Mathematics, University of Alabama}
\and Shuzhong Zhang\footnotemark[1]
}

\date{\today}

\maketitle

\begin{abstract}


In this paper we propose a randomized primal-dual proximal block coordinate updating 
framework for a general multi-block convex optimization model with coupled objective function and linear constraints. Assuming mere convexity, we establish its $O(1/t)$ convergence rate in terms of the objective value and feasibility measure.
The framework includes several existing algorithms as special cases such as a primal-dual method for bilinear saddle-point problems (PD-S),
the proximal Jacobian ADMM (Prox-JADMM) and a randomized variant of the ADMM method for multi-block convex optimization.
Our analysis recovers and/or strengthens the convergence properties of several existing algorithms. For example, for PD-S our result leads to the same order of convergence rate without the previously assumed boundedness condition on the constraint sets, and for Prox-JADMM 
the new result provides convergence rate in terms of the objective value and the feasibility violation.
It is well known that the original ADMM may fail to converge when the number of blocks exceeds two. Our result shows that if an appropriate randomization procedure is invoked to select the updating blocks, then a sublinear rate of convergence in expectation can be guaranteed for multi-block ADMM, without assuming any strong convexity. 
The new approach is also extended to solve problems where only a stochastic approximation of the (sub-)gradient of the objective is available, and we establish an $O(1/\sqrt{t})$ convergence rate of the extended approach for solving stochastic programming.

\vspace{1.5cm}


\vspace{1cm}


\end{abstract}

\newpage

\section{Introduction}

In this paper, we consider the following multi-block structured convex optimization model
\begin{equation}\label{eq:main}
\begin{aligned}
\min_{\vx,\vy}\;\; & f(\vx_1,\cdots,\vx_N)+\sum_{i=1}^N u_i(\vx_i) + g(\vy_1,\cdots,\vy_M)+\sum_{j=1}^M v_j(\vy_j) \\
\st & \sum_{i=1}^N \vA_{i}\vx_i +\sum_{j=1}^M \vB_j\vy_j = \vb\\
& \vx_i\in \Xcal_i,\, i=1,\ldots,N;\ \vy_j\in \Ycal_j,\, j=1,\ldots,M,
\end{aligned}
\end{equation}
where the variables $\vx=(\vx_1;\cdots;\vx_N)$ and $\vy=(\vy_1;\cdots;\vy_M)$ are naturally partitioned into $N$ and $M$ blocks respectively, $\vA=(\vA_1,\cdots,\vA_{N})$ and $\vB=(\vB_1,\cdots,\vB_M)$ are block matrices, $\Xcal_i$'s and $\Ycal_j$'s are some closed convex sets, $f$ and $g$ are smooth convex functions, and $u_i$'s and $v_j$'s are proper closed convex (possibly nonsmooth) functions.

\subsection{Motivating examples}

Optimization problems in the form of \eqref{eq:main} have many emerging applications from various fields. For example, the constrained lasso (classo) problem that was first studied by James \emph{et al.} \cite{james2012constrained} as a generalization of the lasso problem, can be formulated as
\begin{eqnarray}\label{prob:classo}
\begin{array}{ll}
\min\limits_{\vx} & \frac{1}{2}\|A x-b\|_2^2+\tau \|\vx\|_1 \\
\st  & C\vx\leq d,  \\
\end{array}
\end{eqnarray}
where $A\in\RR^{m\times p}$, $b\in\RR^m$ are the observed data, and $C\in\RR^{n\times p}$, $d\in\RR^n$ are the predefined data matrix and vector. Many widely used statistical models can be viewed as special cases of \eqref{prob:classo}, including the monotone curve estimation, fused lasso, generalized lasso, and so on \cite{james2012constrained}. By partitioning the variable $\vx$ into blocks as $\vx=(\vx_1;\cdots;\vx_K)$ where $\vx_i\in\RR^{p_i}$ as well as other matrices and vectors in \eqref{prob:classo} correspondingly, and introducing another slack variable $y$, the classo problem can be transformed to
\begin{eqnarray}\label{prob:classo1}
\begin{array}{ll}
\min\limits_{\vx,\vy} & \frac{1}{2}\left\|\sum\limits_{i=1}^K A_i\vx_i-b\right\|_2^2+\tau \sum\limits_{i=1}^K\|\vx_i\|_1 \\
\st  & \sum\limits_{i=1}^{K}C_i\vx_i+y=d, \,\, y\geq0,  \\
\end{array}
\end{eqnarray}
which is in the form of \eqref{eq:main}.

Another interesting example is the extended linear-quadratic programming \cite{rockafellar1991ext-quadprog} that can be formulated as
\begin{eqnarray}\label{prob:ext-quadprog}
\begin{array}{ll}
\min\limits_x &\frac{1}{2} x^\top P x+a^\top x+\max\limits_{s\in\mathcal{S}}\left\{(d-Cx)^\top s-\frac{1}{2}s^\top Qs\right\},\\
\st & Ax\le b,
\end{array}
\end{eqnarray}
where $P$ and $Q$ are symmetric positive semidefinite matrices, and $\mathcal{S}$ is a polyhedral set. Apparently, \eqref{prob:ext-quadprog} includes quadratic programming as a special case. In general, its objective is a piece-wise linear-quadratic convex function.
Let $g(s)=\frac{1}{2}s^\top Qs+\iota_{\cS}(s)$, where $\iota_{\mathcal{S}}$ denotes the indicator function of $\cS$. Then
$$\max\limits_{s\in \cS}\left\{(d-Cx)^\top s-\frac{1}{2}s^\top Qs\right\}=g^*(d-Cx),$$
where $g^*$ denotes the convex conjugate of $g$. Replacing $d-Cx$ by $y$ and introducing slack variable $z$, we can equivalently write \eqref{prob:ext-quadprog}  into the form of \eqref{eq:main}:
\begin{eqnarray}\label{prob:ext-quadprog1}
\begin{array}{ll}
\min\limits_{x,y,z} &\frac{1}{2} x^\top P x+a^\top x+g^*(y),\\
\st & Ax+z= b,\ z\ge 0,\ Cx+y=d,
\end{array}
\end{eqnarray}
for which one can further partition the $x$-variable into a number of disjoint blocks.

Many other interesting applications in various areas can be formulated as optimization problems in the form of~\eqref{eq:main}, including those arising from signal processing, image processing, machine learning and statistical learning; see \cite{hong2014block,cui2015convergence,chen2015convergence,gao2015first} and the references therein.

Finally, we mention that computing a point on the central path for a generic convex programming in block variables $(x_1;\cdots;x_N)$:
\[
\begin{array}{ll}
\min\limits_x & f(x_1,\cdots,x_N)  \\
\st & \sum_{i=1}^N A_i x_i \le b,\, x_i \ge 0, \, i=1,2,...,N
\end{array}
\]
boils down to
\[
\begin{array}{ll}
\min\limits_{x,y} & f(x_1,\cdots,x_N) - \mu e^\top \ln x - \mu e^\top \ln y \\
\st & \sum_{i=1}^N A_i x_i + y = b,
\end{array}
\]
where $\mu>0$ and $e^\top \ln v$ indicates the sum of the logarithm of all the components of $v$.  This model is again in the form of \eqref{eq:main}.

\subsection{Related works in the literature}

Our work relates to two recently very popular topics: the \emph{Alternating Direction Method of Multipliers} (ADMM) for multi-block structured problems and the first-order primal-dual method for bilinear saddle-point problems. Below we review the two methods and their convergence results. More complete discussion on their connections to our method will be provided after presenting our algorithm.

\subsubsection*{Multi-block ADMM and its variants}
One well-known approach for solving a linear constrained problem in the form of \eqref{eq:main} is the augmented Lagrangian method, which iteratively updates the primal variable $(x,y)$ by minimizing the augmented Lagrangian function in \eqref{eq:aug-func} and then the multiplier $\lambda$ through dual gradient ascent. However, the linear constraint couples $x_1,\ldots,x_N$ and $y_1,\ldots,y_M$ all together, it can be very expensive to minimize the augmented Lagrangian function simultaneously with respect to all block variables. Utilizing the multi-block structure of the problem, the multi-block ADMM updates the block variables sequentially, one at a time with the others  fixed to their most recent values, followed by the update of multiplier. Specifically, it performs the following updates iteratively (by assuming the absence of the coupled functions $f$ and $g$):
\begin{equation}\label{alg:multi-adm}
\left\{
\begin{array}{rcl}
x_1^{k+1} &=& \arg\min_{x_1 \in \Xcal_1} \CL_\rho(x_1,x_2^k,\cdots,x_N^k, y^k, \lambda^k),\\
          &\vdots& \\
x_N^{k+1} &=& \arg\min_{x_N \in \Xcal_N} \CL_\rho(x_1^{k+1},\cdots,x_{N-1}^{k+1},x_N, y^k,\lambda^k),\\
y_1^{k+1} &=& \argmin_{y_1\in \cY_1}\CL_\rho(x^{k+1},y_1,y_2^k,\cdots,y_M^k,\lambda^k),\\
          &\vdots& \\
y_M^{k+1} &=& \argmin_{y_M\in\cY_M}\CL_\rho(x^{k+1},y_1^{k+1},\cdots,y_{M-1}^{k+1}, y_M, \lambda^k),\\
\lambda^{k+1} &=& \lambda^k-\rho(Ax^{k+1}+By^{k+1}-b),\\\end{array}\right.
\end{equation}
where the augmented Lagrangian function is defined as:
\begin{equation}\label{eq:aug-func}
\CL_{\rho}(x,y,\lambda)=\sum_{i=1}^N u_i(x_i)+\sum_{j=1}^M v_j(y_j)-\lambda^{\top}\left(Ax+By-b\right)+\frac{\rho}{2}\left\|Ax+By-b\right\|^2.
\end{equation}

When there are only two blocks, i.e., $N=M=1$, the update scheme in \eqref{alg:multi-adm} reduces to the classic 2-block ADMM \cite{Glowinski1975,gabay1976dual}.
The convergence properties of the ADMM for solving 2-block separable convex problems have been studied extensively. Since the 2-block ADMM can be viewed as a manifestation of some kind of operator splitting, its convergence follows from that of the so-called Douglas-Rachford operator splitting method; see \cite{glowinski1989augmented,eckstein1992douglas}. Moreover, the convergence rate of the 2-block ADMM has been established recently by 
many authors; see e.g.~\cite{journals/siamnum/HeY12,monteiro2010iteration,deng2012global,LMZ2015JORSC,hongluo2012linearadmm,boley2013local}.


Although the multi-block ADMM scheme in \eqref{alg:multi-adm} performs very well for many instances encountered in practice (e.g. \cite{peng2012rasl,tao2011recovering}), it may fail to converge for some instances if there are more than 2 block variables, i.e., $N+M\ge 3$.
In particular, an example was presented in \cite{chen2013direct} to show that the ADMM may even diverge with 3 blocks of variables, when solving a linear system of equations. Thus, some additional assumptions or modifications will have to be in place to ensure 
convergence of the multi-block ADMM.
In fact, by incorporating some extra correction steps or changing the Gauss-Seidel updating rule, \cite{deng2013parallel,he2012convergence,He-Hou-Yuan,he2012alternating, xu2016hybrid-BCU}
show that the convergence can still be achieved for the multi-block ADMM.
Moreover, if some part of the objective function is strongly convex or the objective has certain regularity property, then it can be shown that the convergence holds under various conditions; see
\cite{chencai2013convergence,LMZ2015JORSC,cai2014directstrong,lin2015global,li2015convergent,han2012note,wang2015ncvx-admm}.
Using some other conditions including the error bound condition and taking small dual stepsizes, or by adding some perturbations to the original problem, authors of \cite{hongluo2012linearadmm,lin2015iteration} establish the rate of convergence results even without strong convexity. Not only for the problem with linear constraint, in \cite{chen2015efficient,li2016schur,sun2015convergent} multi-block ADMM are extended to solve convex linear/quadratic conic programming problems. In a very recent work \cite{sun2015expected}, Sun, Luo and Ye propose a randomly permuted ADMM (RP-ADMM) that basically chooses a random permutation of the block indices and performs the ADMM update according to the order of indices in that permutation, 
and they show that the RP-ADMM converges in expectation for solving non-singular square linear system of equations.


In \cite{hong2014block}, the authors propose a block successive upper bound minimization method of multipliers (BSUMM) to solve problem \eqref{eq:main} without $y$ variable. 
Essentially, at every iteration, the BSUMM replaces the nonseparable part $f(x)$ by an upper-bound function and works on that modified function in an ADMM manner. Under some error bound conditions and a diminishing dual stepsize assumption, the authors are able to show that the iterates produced by the BSUMM algorithm converge to the set of primal-dual optimal solutions. Along a similar direction, Cui et al.\ \cite{cui2015convergence} introduces a quadratic upper-bound function for the nonseparable function $f$ to solve 2-block problems; they show that their algorithm has an $O(1/t)$ convergence rate, where $t$ is the number of total iterations. Very recently, \cite{gao2015first} has proposed a set of variants of the ADMM by adding some proximal terms into the algorithm; the authors have managed to prove $O(1/t)$ convergence rate for the 2-block case, and the same results applied for general multi-block case under some strong convexity assumptions. Moreover, \cite{chen2015convergence} shows the convergence of the ADMM for 2-block problems by imposing quadratic structure on the coupled function $f(x)$ and also the convergence of RP-ADMM for multi-block case where all separable functions vanish (i.e. $u_i(x_i)=0,\,\forall i$).

\subsubsection*{Primal-dual method for bilinear saddle-point problems}
Recently, the work \cite{dang2014randomized} generalizes the first-order primal-dual method in \cite{chambolle2011first} to a randomized method for solving a class of saddle-point problems in the following form:
\begin{equation}\label{eq:saddle-prob}
\min_{z\in\cZ}\left\{h(z)+\max_{x\in \cX} \left\langle z, \sum_{i=1}^N A_ix_i\right\rangle-\sum_{i=1}^N u_i(x_i)\right\},
\end{equation}
where $x=(x_1;\ldots;x_N)$ and $\cX=\cX_1\times\cdots\times\cX_N$. Let $\cZ=\RR^p$ and $h(z)=-b^\top z$. Then it is easy to see that \eqref{eq:saddle-prob} is a saddle-point reformulation of the multi-block structured optimization problem
$$\min_{x\in\cX} \sum_{i=1}^N u_i(x_i), \st \sum_{i=1}^N A_ix_i=b,$$
which is a special case of \eqref{eq:main} without $y$ variable or the coupled function $f$.

At each iteration, the algorithm in \cite{dang2014randomized} chooses one block of $x$-variable uniformly at random and performs a proximal update to it, followed by another proximal update to the $z$-variable. More precisely, it iteratively performs the updates:
\begin{subequations}\label{alg:r1st-pd}
\begin{align}
&x_i^{k+1}=\left\{
\begin{array}{ll}
\argmin_{x_i\in\cX_i} \langle -\bar{z}^k, A_ix_i\rangle + u_i(x_i)+\frac{\tau}{2}\|x_i-x_i^k\|_2^2, & \text{ if }i=i_k, \label{alg:r1st-pd-x}\\
x_i^k, & \text{ if } i\neq i_k,
\end{array}\right.\\
&z^{k+1}=\argmin_{z\in\cZ} h(z)+\langle z, A x^{k+1}\rangle+\frac{\eta}{2}\|z-z^k\|_2^2, \label{alg:r1st-pd-z}\\
&\bar{z}^{k+1}=q(z^{k+1}-z^k)+z^{k+1}, \label{alg:r1st-pd-zbar}
\end{align}
\end{subequations}
where $i_k$ is a randomly selected block, and $\tau,\eta$ and $q$ are certain parameters\footnote{Actually, \cite{dang2014randomized} presents its algorithm in a more general way with the parameters adaptive to the iteration. However, its convergence result assumes constant values of these parameters for the weak convexity case.}. When there is only one block of $x$-variable, i.e., $N=1$, the scheme in \eqref{alg:r1st-pd} becomes exactly the primal-dual method in \cite{chambolle2011first}. Assuming the boundedness of the constraint sets $\cX$ and $\cZ$, \cite{dang2014randomized} shows that under weak convexity, $O(1/t)$ convergence rate result of the scheme can be established by choosing appropriate parameters, and if $u_i$'s are all strongly convex, the scheme can be accelerated to have $O(1/t^2)$ convergence rate by adapting the parameters.

\subsection{Contributions and organization of this paper}
\begin{itemize}
\item We propose a randomized primal-dual coordinate update algorithm to solve problems in the form of \eqref{eq:main}.
The key feature is to introduce randomization as done in \eqref{alg:r1st-pd} to the multi-block ADMM framework \eqref{alg:multi-adm}. Unlike the random permutation scheme as previously investigated in \cite{sun2015expected,chen2015convergence}, we simply choose a subset of blocks of variables based on the uniform distribution. In addition, we perform a proximal update to that selected subset of variables. With appropriate proximal terms (e.g., the setting in \eqref{matPQ}), the selected block variables can be decoupled, and thus the updates can be done in parallel.
\item More general than \eqref{alg:multi-adm}, we can accommodate coupled terms in the objective function in our algorithm by linearizing such terms. By imposing Lipschitz continuity condition on the partial gradient of the coupled functions $f$ and $g$ and using proximal terms, we show that our method has an expected $O(1/t)$ convergence rate for solving problem \eqref{eq:main} under mere convexity assumption. 
\item We show that our algorithm includes several existing methods as special cases such as the scheme in \eqref{alg:r1st-pd} and the proximal Jacobian ADMM in \cite{deng2013parallel}. Our result indicates that the $O(1/t)$ convergence rate of the scheme in \eqref{alg:r1st-pd} can be shown without assuming boundedness of the constraint sets. In addition,  the same order of convergence rate of the proximal Jacobian ADMM can be established in terms of a better measure.
\item Furthermore, the linearization scheme allows us to deal with stochastic objective function, for instance, when the function $f$ is given in a form of expectation $f=\EE_\xi[f_\xi(x)]$ where $\xi$ is a random vector. As long as an unbiased estimator of the (sub-)gradient of $f$ is available, we can extend our method to the stochastic problem and an expected $O(1/\sqrt{t})$ convergence rate is achievable.
\end{itemize}

The rest of the paper is organized as follows. In Section \ref{sec:Prelim}, we introduce our algorithm and present some preliminary results.
In Section \ref{sec:Sublinear-Rate}, we present the sublinear convergence rate results of the proposed algorithm.
Depending on the multi-block structure of $\vy$, different conditions and parameter settings are presented in Subsections \ref{subsection:no-y}, \ref{sbsec:mulxsiny} and  \ref{sbsec:mulxy}, respectively.
In Section~\ref{sbsec:sto-conv}, we present an extension of our algorithm where the objective function is assumed not to be even exactly computable, instead only some first-order stochastic approximation is available. The convergence analysis is extended to such settings accordingly. Numerical results are shown in Section~\ref{sec:numerical}.
In Section~\ref{sec:connection}, we discuss the connections of our algorithm to other well-known methods in the literature. Finally, we conclude the paper in Section~\ref{sec:conc-rem}.
The proofs for the technical lemmas are presented in Appendix~\ref{sec:app-A}, and the proofs for the main theorems are in Appendix~\ref{sec:app-B}.

\section{Randomized Primal-Dual Block Coordinate Update Algorithm}
\label{sec:Prelim}

In this section, we first present some notations and then introduce our algorithm as well as some preliminary lemmas.

\subsection{Notations}
We denote $\cX=\cX_1\times\cdots\times \cX_N$ and $\cY=\cY_1\times\cdots\times \cY_M$. For any symmetric positive semidefinite matrix $W$, we define $\|z\|_W=\sqrt{z^\top W z}$. 
Given an integer $\ell>0$, $[\ell]$ denotes the set $\{1,2,\cdots,\ell\}$. We use $I$ and $J$ as index sets, 
while $I$ is also used to denote the identity matrix; we believe that the intention is evident in the context. Given $I=\{i_1,i_2,\cdots,i_n\}$, 
we denote:
\begin{itemize}
\item Block-indexed variable: $\vx_I=(\vx_{i_1};\vx_{i_2};\cdots;\vx_{i_n})$;
\item Block-indexed set: $\cX_I=\cX_{i_1}\times\cdots\times \cX_{i_n}$;
\item Block-indexed function: $u_I(\vx_I)=u_{i_1}(\vx_{i_1})+u_{i_2}(\vx_{i_2})+\cdots+u_{i_n}(\vx_{i_n})$;
\item Block-indexed gradient: $\nabla_{I}f(x)=(\nabla_{i_1}f(x);\nabla_{i_2}f(x);\cdots;\nabla_{i_n}f(x))$;
\item Block-indexed matrix: $A_I=\big[A_{i_1}, A_{i_2},\cdots,  A_{i_n} \big]$.
\end{itemize}

\subsection{Algorithm}
Our algorithm is rather general. Its major ingredients are randomization in selecting block variables, linearization of the coupled functions $f$ and $g$, and adding proximal terms. Specifically, at each iteration $k$, it first randomly samples a subset $I_k$ of blocks of $\vx$, and then a subset $J_k$ of blocks of $\vy$ according to the uniform distribution over the indices.
The randomized sampling rule is as follows:

\begin{quote}
{\bf Randomization Rule (U):} For the given integers $n\le N$ and $m\le M$, it randomly chooses index sets $I_k\subset[N]$ with  $|I_k|=n$ and $J_k\subset[M]$ with $|J_k|=m$ {\em uniformly}; i.e., for any subsets $\{i_1,i_2,\ldots,i_n\}\subset [N]$ and $\{j_1,j_2,\ldots,j_m\}\subset [M]$, the following holds
\begin{align}
\Prob\big[I_k=\{i_1,i_2,\ldots,i_n\}\big]=1/\left(\begin{array}{c}N\\n\end{array}\right),\cr
\Prob\big[J_k=\{j_1,j_2,\ldots,j_m\}\big]=1/\left(\begin{array}{c}M\\m\end{array}\right).\nonumber
\end{align}
\end{quote}
After those subsets have been selected, it performs a prox-linear update to those selected blocks based on the augmented Lagrangian function, followed by an update of the Lagrangian multiplier.
 The details of the method are summarized in Algorithm \ref{alg:rpdc} below. 

\begin{algorithm}\caption{\textbf{R}andomized \textbf{P}rimal-\textbf{D}ual \textbf{B}lock Coordinate \textbf{U}pdate Method (RPDBU)}\label{alg:rpdc}
\DontPrintSemicolon
{\small\textbf{Initialization:} choose $\vx^0, \vy^0$ and $\vlam^0=0$; let $\vr^0=\vA\vx^0+\vB\vy^0-\vb$; choose $\rho, \rho_x, \rho_y$\;
\For{$k=0,1,\ldots$}{
Randomly select $I_k\subset[N]$ and $J_k\subset[M]$ with $|I_k|=n$ and $|J_k|=m$ according to {\bf (U)}.\;
Let $\vx_i^{k+1}=\vx_i^k,\,\forall i\not\in I_k$ and $\vy_j^{k+1}=\vy_j^k,\,\forall j\not\in J_k$. \;
For $I = I_k$, perform the update
\begin{align}
&\vx_I^{k+1}= \argmin_{\vx_I\in \Xcal_I}\langle \nabla_I f(\vx^k)-\vA_I^\top \vlam^k, \vx_I\rangle+u_I(\vx_I)+\frac{\rho_x}{2}\|\vA_I(\vx_I-\vx_I^k)+\vr^k\|^2+\frac{1}{2}\|\vx_I-\vx_I^k\|_{\vP^k}^2,\label{eq:update-x}\\
&\vr^{k+\frac{1}{2}}=\vr^k+\vA_{I}(\vx_{I}^{k+1}-\vx_{I}^k). \label{eq:update-r1}
\end{align}
For $J = J_k$, perform the update
\begin{align}
&\vy_J^{k+1}= \argmin_{\vy_J\in \Ycal_J}\langle \nabla_J g(\vy^k)-\vB_J^\top \vlam^k, \vy_J\rangle+v_J(\vy_J)+\frac{\rho_y}{2}\|\vB_J(\vy_J-\vy_J^k)+\vr^{k+\frac{1}{2}}\|^2+\frac{1}{2}\|\vy_J-\vy_J^k\|_{\vQ^k}^2,\label{eq:update-y}\\
&\vr^{k+1}=\vr^{k+\frac{1}{2}}+\vB_J(\vy_J^{k+1}-\vy_J^k). \label{eq:update-r2}
\end{align}
Update the multiplier by
\begin{equation}\label{eq:update-lam}
\vlam^{k+1}=\vlam^k-\rho\vr^{k+1}.
\end{equation}
}
}
\end{algorithm}
\normalsize

In Algorithm \ref{alg:rpdc}, $\vP^k$ and $\vQ^k$ are predetermined positive semidefinite matrices with appropriate dimensions.
For the selected blocks in $I_k$ and $J_k$, instead of implementing the exact minimization of the augmented Lagrangian function, we perform a block proximal gradient update. In particular, before minimization, we first linearize the coupled functions $f$, $g$, and add some proximal terms to it. Note that one can always select all blocks, i.e., $I_k=[N]$ and $J_k=[M]$. Empirically however, the block coordinate update method usually outperforms the full coordinate update method if the problem possesses certain structures; see \cite{peng2016cf} for an example. In addition, by choosing appropriate $P^k$ and $Q^k$, the problems \eqref{eq:update-x} and \eqref{eq:update-y} can be separable with respect to the selected blocks, and thus one can update the variables in parallel.

\subsection{Preliminaries}
Let $\vw$ be the aggregated primal-dual variables and $H(w)$ the primal-dual linear mapping; namely
\begin{equation}\label{w-H}
\vw=\left(\begin{array}{c}
\vx\\ \vy \\ \vlam
\end{array}\right),\quad H(\vw)=\left(\begin{array}{c}-\vA^\top\vlam\\
-\vB^\top\vlam\\ \vA\vx+\vB\vy-\vb\end{array}\right),
\end{equation}
and also let
\begin{align*}
&u(\vx)=\sum_{i=1}^N u_i(\vx_i),\quad v(\vy)=\sum_{j=1}^M v_j(\vy_j),\\
&F(\vx)=f(\vx)+u(\vx),\quad G(\vy)=g(\vy)+v(\vy),\quad \Phi(\vx,\vy)=F(\vx)+G(\vy).
\end{align*}
The point $(\vx^*,\vy^*)$ is a solution to \eqref{eq:main} \emph{if and only if} there exists $\vlam^*$ such that
\begin{subequations}\label{sol-cond}
\begin{align}
&\Phi(\vx,\vy)-\Phi(\vx^*,\vy^*)+(\vw-\vw^*)^\top H(\vw^*)\ge 0,\,\forall (\vx,\vy) \in \Xcal\times \Ycal, \,\forall \vlam,\label{eq:1st-opt}\\
& \vA\vx^*+\vB\vy^*=\vb\label{eq:feas}\\
& \vx^*\in \Xcal, \quad \vy^*\in \Ycal.\label{eq:feas2}
\end{align}
\end{subequations}


The following lemmas will be used in our subsequent analysis, whose proofs are elementary and thus are omitted here.
\begin{lemma}
For any two vectors $\vw$ and $\tilde{\vw}$, it holds
\begin{equation}\label{equivHw}(\vw-\tilde{\vw})^\top H(\vw)=(\vw-\tilde{\vw})^\top H(\tilde{\vw}).
\end{equation}
\end{lemma}

\begin{lemma}
For any two vectors $\vu, \vv$ and a positive semidefinite matrix $\vW$:
\begin{equation}\label{uv-cross}
\vu^\top \vW \vv = \frac{1}{2}\big(\|\vu\|_\vW^2+\|\vv\|_\vW^2-\|\vu-\vv\|_\vW^2\big).
\end{equation}
\end{lemma}

\begin{lemma}
For any nonzero positive semidefinite matrix $\vW$, it holds for any $\vz$ and $\hat{\vz}$ of appropriate size that
\begin{equation}\label{z-posW}
\|\vz-\hat{\vz}\|^2\ge \frac{1}{\|\vW\|_2}\|\vz-\hat{\vz}\|_{\vW}^2,
\end{equation}
where $\|W\|_2$ denotes the matrix operator norm of $W$.
\end{lemma}

The following lemma presents a useful property of $H(\vw)$, which essentially follows from \eqref{equivHw}.
\begin{lemma}\label{property-on-H}
For any vectors $w^0,w^1,\ldots,w^{t}$, and sequence of positive numbers $\beta^0,\beta^1,\ldots,\beta^t$,  it holds that
\begin{equation}\label{prop-mas-H}
\left(\frac{\sum\limits_{k=0}^t\beta^k\vw^k}{\sum\limits_{k=0}^t\beta^k}-w\right)^{\top}H\left(\frac{\sum\limits_{k=0}^t\beta^k\vw^k}{\sum\limits_{k=0}^t\beta^k}\right)
=\frac{1}{\sum\limits_{k=0}^t\beta^k}\sum\limits_{k=0}^t\beta^k(\vw^k-\vw)^{\top}H(\vw^k).
\end{equation}
\end{lemma}


\section{Convergence Rate Results}
\label{sec:Sublinear-Rate}

In this section, we establish sublinear convergence rate results of Algorithm \ref{alg:rpdc} for three different cases. We differentiate those cases based on whether or not $\vy$ in problem \eqref{eq:main} also has the multi-block structure. In the first case where $\vy$ is a multi-block variable, it requires $\frac{n}{N}=\frac{m}{M}$ where $n$ and $m$ are the cardinalities of the subsets of $\vx$ and $\vy$ selected in our algorithm respectively. Since the analysis only requires weak convexity, we can ensure the condition to hold by adding \emph{zero} component functions if necessary, in such a way that $N=M$ and then choosing $n=m$. The second case is that $\vy$ is treated as a single-block variable, and this can be reflected in our algorithm by simply selecting all $\vy$-blocks every time, i.e.\ $m=M$. The third case assumes no $y$-variable at all. It falls into the first and second cases, and we discuss this case separately since it requires weaker conditions to guarantee the same convergence rate.

Throughout our analysis, we choose the matrices $P^k$ and $Q^k$ in Algorithm \ref{alg:rpdc} as follows:
\begin{equation}\label{matPQ}
\vP^k=\hat{\vP}_{I_k}-\rho_x\vA_{I_k}^\top\vA_{I_k}, \qquad \vQ^k=\hat{\vQ}_{J_k}-\rho_y\vB_{J_k}^\top\vB_{J_k},
\end{equation}
where $\hat{\vP}$ and $\hat{\vQ}$ are given symmetric positive semidefinite and block diagonal matrices, and $\hat{\vP}_{I_k}$ denotes the diagonal blocks of $\hat{\vP}$ indexed by $I_k$. Note that such choice of $P^k$ and $Q^k$ makes the selected block variables $x_{I_k}$ in \eqref{eq:update-x} and $y_{J_k}$ in \eqref{eq:update-y} decoupled, and thus both updates can be computed in parallel.
In addition, we make the following assumptions:
\begin{assumption}[Convexity]\label{assump1}
For \eqref{eq:main}, $\Xcal_i$'s and $\Ycal_j$'s are some closed convex sets, $f$ and $g$ are smooth convex functions, and $u_i$'s and $v_j$'s are proper closed convex
function. 
\end{assumption}

\begin{assumption}[Existence of a solution]\label{assump2}
There is at least one point $\vw^*=(\vx^*,\vy^*,\vlam^*)$ satisfying the conditions in \eqref{sol-cond}.
\end{assumption}

\begin{assumption}[Lipschitz continuous partial gradient]\label{assump3}
There exist constants $L_f$ and $L_g$ such that for any subset $I$ of $[N]$ with $|I|=n$ and any subset $J$ of $[M]$ with $|J|=m$, it holds that
\begin{subequations}
\begin{align}
&\|\nabla_I f(\vx+\vU_I\tilde{\vx})-\nabla_I f(\vx)\|\le L_f\|\tilde{\vx}_I\|,\,\forall \vx, \tilde{\vx},\\
&\|\nabla_J g(\vy+\vU_J\tilde{\vy})-\nabla_J g(\vy)\|\le L_g\|\tilde{\vy}_J\|,\,\forall \vy, \tilde{\vy},
\end{align}
\end{subequations}
where $U_I \tilde{x}$ keeps the blocks of $\tilde{x}$ that are indexed by $I$ and zero elsewhere. 
\end{assumption}

Before presenting the main convergence rate result, we first establish a few key lemmas.
\begin{lemma}[One-step analysis]\label{lem:1step}
Let $\{(\vx^k,\vy^k, \vr^k,\vlam^k)\}$ be the sequence generated from Algorithm \ref{alg:rpdc} with matrices $\vP^k$ and $\vQ^k$ defined as in \eqref{matPQ}.
Then the following inequalities hold
\begin{eqnarray}\label{eq:1step-x}
& & \EE_{I_k}\left[F(\vx^{k+1})-F(\vx)+(\vx^{k+1}-\vx)^\top(-\vA^\top\vlam^{k+1})\right.\cr
& &\hspace{0.8cm}
+\left.(\rho_x-\rho)(\vx^{k+1}-\vx)^\top\vA^\top\vr^{k+1}
-\rho_x(\vx^{k+1}-\vx)^\top\vA^\top\vB(\vy^{k+1}-\vy^k)\right]\cr
& & +\EE_{I_k}(\vx^{k+1}-\vx)^\top(\hat{\vP}-\rho_x A^\top A)(\vx^{k+1}-\vx^k)-\frac{L_f}{2}\EE_{I_k}\|\vx^k-\vx^{k+1}\|^2\cr
& \le & \left(1-\frac{n}{N}\right)\big[F(\vx^k)-F(\vx)+(\vx^{k}-\vx)^\top(-\vA^\top\vlam^k)+
\rho_x(\vx^{k}-\vx)^\top\vA^\top\vr^{k}\big], \label{ineq-k1-x}
\end{eqnarray}
and
\begin{eqnarray}
& & \EE_{J_k}\big[G(\vy^{k+1})-G(\vy)+(\vy^{k+1}-\vy)^\top(-\vB^\top\vlam^{k+1})+
(\rho_y-\rho)(\vy^{k+1}-\vy)^\top\vB^\top\vr^{k+1}\big]\cr
& & +\EE_{J_k}(\vy^{k+1}-\vy)^\top(\hat{\vQ}-\rho_y B^\top B)(\vy^{k+1}-\vy^k)-\frac{L_g}{2}\EE_{J_k}\|\vy^k-\vy^{k+1}\|^2\cr
& &-
\left(1-\frac{m}{M}\right)\rho_y(\vy^{k}-\vy)^\top\vB^\top\vA(\vx^{k+1}-\vx^k)\cr
& \le & \left(1-\frac{m}{M}\right)\left[G(\vy^k)-G(\vy)+(\vy^{k}-\vy)^\top(-\vB^\top\vlam^k)+
\rho_y(\vy^{k}-\vy)^\top\vB^\top\vr^{k}\right], \label{ineq-k1-y}
\end{eqnarray}
where $\EE_{I_k}$ denotes expectation over $I_k$ and conditional on all previous history.
\end{lemma}

Note that for any feasible point $(x,y)$ (namely, $x\in\cX, y\in\cY$ and $Ax+By=b$),
\begin{eqnarray}
\vA\vx^{k+1}-\vA \vx &=&\frac{1}{\rho}(\vlam^k-\vlam^{k+1})-(\vB\vy^{k+1}-\vb)+(\vB\vy-\vb)\cr
&=& \frac{1}{\rho}(\vlam^k-\vlam^{k+1})-\vB(\vy^{k+1}-\vy) \label{x-lam-cross}
\end{eqnarray}
and
\begin{eqnarray}\label{y-lam-cross}
\vB\vy^{k}-\vB \vy &=&\frac{1}{\rho}(\vlam^{k-1}-\vlam^{k})-(\vA\vx^{k}-\vb)+(\vA\vx-\vb)\cr
&=& \frac{1}{\rho}(\vlam^{k-1}-\vlam^{k})-\vA(\vx^{k}-\vx).
\end{eqnarray}
Then, using \eqref{uv-cross} we have the following result.

\begin{lemma}\label{lem:1step-ba}
For any feasible point $(\vx,\vy)$ and integer $t$, it holds
\begin{eqnarray}\label{feas-x}
& & \sum_{k=0}^t (\vx^{k+1}-\vx)^\top\vA^\top\vB(\vy^{k+1}-\vy^k)\\
&=&\frac{1}{\rho}\sum_{k=0}^t (\vlam^k-\vlam^{k+1})^\top\vB(\vy^{k+1}-\vy^k)-\frac{1}{2}\left(\|\vy^{t+1}-\vy\|_{\vB^\top\vB}^2-\|\vy^{0}-\vy\|_{\vB^\top\vB}^2
+\sum_{k=0}^t\|\vy^{k+1}-\vy^k\|_{\vB^\top\vB}^2\right)\nonumber
\end{eqnarray}
and
\begin{eqnarray}\label{feas-y}
& &\sum_{k=0}^t(\vy^{k}-\vy)^\top\vB^\top\vA(\vx^{k+1}-\vx^k)\\
&=&\frac{1}{\rho}\sum_{k=0}^t(\vlam^{k-1}-\vlam^{k})^\top\vA(\vx^{k+1}-\vx^k)+\frac{1}{2}\left(\|\vx^0-\vx\|_{\vA^\top\vA}^2-\|\vx^{t+1}-\vx\|_{\vA^\top\vA}^2
+\sum_{k=0}^t\|\vx^{k+1}-\vx^k\|_{\vA^\top\vA}^2\right). \nonumber
\end{eqnarray}
\end{lemma}

\begin{lemma}\label{lem:xy-rate}
Given a continuous function $h$, for a random vector $\hat{w}=(\hat{x},\hat{y},\hat{\vlam})$, if for any feasible point $w=(x,y,\vlam)$ that may depend on $\hat{w}$, 
we have
\begin{equation}\label{eq:xy-phi}\EE\big[\Phi(\hat{x},\hat{y})-\Phi(x,y)+(\hat{w}-w)^\top H(w)\big]\le \EE [h(w)],
\end{equation}
then for any $\gamma>0$ and any optimal solution $(x^*,y^*)$ to \eqref{eq:main} we also have
$$\EE\big[\Phi(\hat{x},\hat{y})-\Phi(x^*,y^*)+\gamma\|A\hat{x}+B\hat{y}-b\|\big]\le \sup_{\|\vlam\|\le \gamma}h(x^*,y^*,\vlam).$$
\end{lemma}

Noting $$\Phi(\vx,\vy)-\Phi(\vx^*,\vy^*)+(\vw-\vw^*)^\top H(\vw^*)=\Phi(\vx,\vy)-\Phi(\vx^*,\vy^*)-(\lambda^*)^\top (Ax+By-b),$$ we can easily show the following lemma by the optimality of $(x^*,y^*,\lambda^*)$ and the Cauchy-Schwarz inequality.
\begin{lemma}\label{lem:lb-obj}
Assume $(x^*,y^*,\lambda^*)$ satisfies \eqref{sol-cond}. Then for any point $(\hat{x},\hat{y})\in\cX\times\cY$, we have
\begin{equation}\label{eq:bd-phi-by-fea}\Phi(\hat{x},\hat{y})-\Phi(x^*,y^*)\ge -\|\lambda^*\|\cdot\|A\hat{x}+B\hat{y}-b\|.
\end{equation}
\end{lemma}

The following lemma shows a connection between different convergence measures, and it can be simply proved by using \eqref{eq:bd-phi-by-fea}. If both $w$ and $\hat{w}$ are deterministic, it reduces to Lemma 2.4 in \cite{gao2014information}.

\begin{lemma} \label{equiv-rate}
Suppose that
\[
\EE\big[\Phi(\hat{x},\hat{y})-\Phi(x^*,y^*)+\gamma\|A\hat{x}+B\hat{y}-b\|\big] \le \epsilon.
\]
Then, we have
\[
\EE\|A\hat{x}+B\hat{y}-b\|\leq \frac{\epsilon}{\gamma-\|\lambda^*\|} \mbox{ and }  -\frac{\|\lambda^*\|\epsilon}{\gamma-\|\lambda^*\|}\le\EE\big[\Phi(\hat{x},\hat{y})-\Phi(x^*,y^*)\big] \leq \epsilon,
\]
where $(x^*,y^*,\lambda^*)$ satisfies the optimality conditions in \eqref{sol-cond}, and we assume $\|\lambda^*\|< \gamma$.
\end{lemma}

The convergence analysis for Algorithm~\ref{alg:rpdc} requires slightly different parameter settings under different structures. In fact, the underlying analysis and results also differ. To account for the differences, we present in the next three subsections the corresponding convergence results. { The first one assumes there is no $y$ part at all; 
the second case assumes a single block on the $y$ side; 
the last one deals with the general case where the ratios $n/N$ is assumed to be equal to $m/M$. 
}

\subsection{Multiple $\vx$ blocks and no $y$ variable} \label{subsection:no-y}
We first consider a special case with no $y$-variable, namely, $g=v=0$ and $B=0$ in \eqref{eq:main}. { This case has its own importance. It is a parallel block coordinate update version of the linearized augmented Lagrangian method (ALM).}

\begin{theorem}[Sublinear ergodic convergence I]\label{thm:rate-3X}
Assume $g(y)=0, v_j(y_j)=0,\,\forall j$ and $B=0$ in \eqref{eq:main}. Let $\{(\vx^k,\vy^k, \vlam^k)\}$ be the sequence generated from Algorithm \ref{alg:rpdc} with 
$\vy^k\equiv y^0$. Assume $\frac{n}{N}=\theta,\,\rho=\theta\rho_x,$
and
\begin{equation}\label{para-mat-3X}
{\vP}^k\succeq L_f\vI,\,\forall k.\end{equation}
Let
\begin{equation}\label{avg-pt-sy-X}\hat{\vx}^{t}=\frac{\vx^{t+1}+\theta\sum_{k=1}^t\vx^k}{1+\theta t}.
\end{equation}
Then, under Assumptions~\ref{assump1}, \ref{assump2} and \ref{assump3}, we have 
{ \begin{eqnarray}\label{eq:rate-iym-X}
& &\max\left\{\left|\EE\big[F(\hat{\vx}^t)-F(\vx^*)\big]\right|,\, \EE\|A\hat{x}^t-b\|\right\}\\
&\le &\frac{1}{1+\theta t}\left[(1-\theta)\left(F(\vx^0)-F(\vx^*)+\rho_x\|r^0\|^2\right) +\frac{1}{2}\|x^0-x^*\|_{\tilde{P}}^2+\frac{\max\{(1+\|\lambda^*\|)^2, 4\|\lambda^*\|^2\}}{2\rho_x}\right]\notag 
\end{eqnarray}
}
where $\tilde{P}=\hat{P}-\rho_x A^\top A$, and $(x^*,\lambda^*)$ is an arbitrary primal-dual solution.
\end{theorem}

{If there is no coupled function $f$, \eqref{para-mat-3X} indicates that we can even choose $P^k=0$ for all $k$, i.e.\ without proximal terms at all. 
Some caution is required here: in that case when $I_k=[N],\forall k$, the algorithm is {\it not}\/ the Jacobian ADMM as discussed in \cite{He-Hou-Yuan} since the block variables are still coupled in the augmented Lagrangian function. To make it parallelizable, a proximal term is needed. Then our result recovers the convergence of the proximal Jacobian ADMM introduced in \cite{deng2013parallel}. In fact, the above theorem strengthens the convergence result in \cite{deng2013parallel} by establishing an $O(1/t)$ rate of convergence in terms of the feasibility measure and the objective value.}

\subsection{Multiple $\vx$ blocks and a single $\vy$ block}\label{sbsec:mulxsiny}
When the $y$-variable is simple to update, it could be beneficial to renew the whole of it at every iteration, such as the problem \eqref{prob:classo1}. In this subsection, we consider the case that there are multiple $\vx$-blocks but a single $\vy$-block (or equivalently, $m=M$), and we establish a sublinear convergence rate result with a different technique of dealing with the $y$-variable.

\begin{theorem}[Sublinear ergodic convergence II]\label{thm:rate-1Yw}
Let $\{(\vx^k,\vy^k, \vlam^k)\}$ be the sequence generated from Algorithm \ref{alg:rpdc} with $m=M$ and $\rho=\rho_y=\theta\rho_x$, where $\theta=\frac{n}{N}=\theta.$ Assume

\begin{equation}\label{para-mat-1Yw}
\hat{\vP}\succeq L_f\vI+\rho_x A^\top A, \qquad
\hat{\vQ}\succeq \frac{L_g}{\theta}\vI+\left(\frac{\rho}{\theta^4}-\frac{\rho}{\theta^2}+\rho_y\right)\vB^\top\vB.
\end{equation}
Let
\begin{equation}\label{avg-pt-sy}\hat{\vx}^{t}=\frac{\vx^{t+1}+\theta\sum_{k=1}^t\vx^k}{1+\theta t},
\quad \hat{\vy}^{t}=\frac{\tilde{\vy}^{t+1}+\theta\sum_{k=1}^t\vy^k}{1+\theta t}
\end{equation}
where
\begin{equation}\tilde{\vy}^{t+1}=\argmin_{\vy\in \Ycal}\langle\nabla g(\vy^t)-\vB^\top\vlam^t, \vy\rangle +v(\vy)+\frac{\rho_x}{2}\|\vA\vx^{t+1}+\vB\vy-\vb\|^2+\frac{\theta}{2}\|\vy-\vy^t\|_{\hat{\vQ}-\rho_y B^\top B}^2.
\end{equation}
Then, under Assumptions~\ref{assump1}, \ref{assump2} and \ref{assump3}, we have 
{ \begin{eqnarray}\label{eq:rate-iym}
& &\max\left\{\left|\EE\big[\Phi(\hat{\vx}^t,\hat{\vy}^t)-\Phi(\vx^*,\vy^*)\big]\right|,\,\EE\|A\hat{x}^t+B\hat{y}^t-b\|\right\}\\
& \le &\frac{1}{1+\theta t}\left[(1-\theta)\left(\Phi(\vx^0,\vy^0)-\Phi(\vx^*,\vy^*)+\frac{\rho_x}{2}\|\vr^{0}\|^2\right)+\frac{1}{2}\|x^0-x^*\|_{\hat{P}}^2+\frac{1}{2}\|y^0-y^*\|_{\theta\tilde{Q}+\rho_x B^\top B}^2\right.\notag\\ 
& &\hspace{1.5cm} +\left.\frac{\max\{(1+\|\lambda^*\|)^2, 4\|\lambda^*\|^2\}}{2\rho_x}\right] \notag
\end{eqnarray}
}
where $\tilde{Q}=\hat{Q}-\rho_y B^\top B$, and $(x^*,y^*,\lambda^*)$ is an arbitrary primal-dual solution.
\end{theorem}
\begin{remark}
It is easy to see that if $\theta=1$, the result in Theorem \ref{thm:rate-1Yw} becomes exactly the same as that in Theorem \ref{thm:rate-cvx} below. In general, they are different because the conditions in \eqref{para-mat-1Yw} on $\hat{P}$ and $\hat{Q}$ are different from those in \eqref{para-mat}.
\end{remark}

\subsection{Multiple $\vx$ and $\vy$ blocks}\label{sbsec:mulxy}
In this subsection, we consider the most general case where both $\vx$ and $\vy$ have multi-block structure. Assuming $\frac{n}{N}=\frac{m}{M}$, we can still have the $O(1/t)$ convergence rate. The assumption can be made without losing generality, e.g., by adding zero components if necessary (which is essentially equivalent to varying the probabilities of the variable selection).
\begin{theorem}[Sublinear ergodic convergence III]\label{thm:rate-cvx}
Let $\{(\vx^k,\vy^k, \vlam^k)\}$ be the sequence generated from Algorithm \ref{alg:rpdc} with the parameters satisfying
\begin{align}
&\rho=\frac{n\rho_x}{N}=\frac{m\rho_y}{M}>0.\label{para-rho}
\end{align}
Assume $\frac{n}{N}=\frac{m}{M}=\theta,$
and $\hat{P}, \hat{Q}$ satisfy one of the following conditions
\begin{subequations}\label{matPQhat}
\begin{align}\label{para-mat}
&\hat{\vP}\succeq (2-\theta)\left(\frac{1-\theta}{\theta^2}+1\right)\rho_x\vA^\top\vA+L_f\vI, \qquad
\hat{\vQ}\succeq \frac{(2-\theta)}{\theta^2}\rho_y\vB^\top\vB+L_g\vI.\\
&\hat{P}_i\succeq(2-\theta)\left(\frac{1-\theta}{\theta^2}+1\right)n\rho_x\vA_i^\top\vA_i+L_f\vI,\,\forall i,\quad \hat{\vQ}_j\succeq \frac{(2-\theta)}{\theta^2}m\rho_y\vB_j^\top\vB_j+L_g\vI,\,\forall j. \label{para-mat-ij}
\end{align}
\end{subequations}
Let
\begin{equation}\label{avg-pt}\hat{\vx}^{t}=\frac{\vx^{t+1}+\theta\sum_{k=1}^t\vx^k}{1+\theta t},\quad \hat{\vy}^{t}=\frac{\vy^{t+1}+\theta\sum_{k=1}^t\vy^k}{1+\theta t}.
\end{equation}
Then, under Assumptions~\ref{assump1}, \ref{assump2} and \ref{assump3}, we have 
{ \begin{eqnarray}\label{eq:rate-cvx}
& &\max\left\{\left|\EE\big[\Phi(\hat{\vx}^t,\hat{\vy}^t)-\Phi(\vx^*,\vy^*)\big]\right|,\,\EE\|A\hat{x}^t+B\hat{y}^t-b\|\right\}\\
& \le &\frac{1}{1+\theta t}\left[(1-\theta)\left(\Phi(\vx^0,\vy^0)-\Phi(\vx^*,\vy^*)+\rho_x\|r^0\|^2\right) +\frac{1}{2}\|x^0-x^*\|_{\tilde{P}}^2+\frac{1}{2}\|y^0-y^*\|_{\hat{Q}}^2\right.\nonumber\\
&&\hspace{1.5cm}\left.+\frac{\max\{(1+\|\lambda^*\|)^2, 4\|\lambda^*\|^2\}}{2\rho_x}\right] \nonumber
\end{eqnarray}
}
where $\tilde{P}=\hat{P}-\theta\rho_x A^\top A$, and $(x^*,y^*,\lambda^*)$ is an arbitrary primal-dual solution.
\end{theorem}


\begin{remark}
When $N=M=1$, the two conditions in \eqref{matPQhat} become the same. However, in general, neither of the two conditions in \eqref{matPQhat} implies the other one. Roughly speaking, for the case of $n\approx N$ and $m\approx M$, the one in \eqref{para-mat} can be weaker, and for the case of $n\ll N$ and $m\ll M$, the one in \eqref{para-mat-ij} is more likely weaker. In addition, \eqref{para-mat-ij} provides an explicit way to choose block diagonal $\hat{P}$ and $\hat{Q}$ by simply setting $\hat{P}_i$ and $\hat{Q}_j$'s to the lower bounds there.
\end{remark}





\section{Randomized Primal-Dual Coordinate Approach for Stochastic Programming}
\label{sbsec:sto-conv}

In this section, we extend our method to solve a stochastic optimization problem where the objective function involves an expectation. Specifically, we assume the coupled function to be in the form of $f(x)=\EE_\xi f_\xi(x)$ where $\xi$ is a random vector. For simplicity we assume $g=v=0$, namely, we consider the following problem
\begin{equation}\label{eq:main2}
\begin{aligned}
\min_\vx\;\; &\EE_\xi f_\xi(\vx)+\sum_{i=1}^N u_i(\vx_i), \\\st & \sum_{i=1}^N \vA_i\vx_i=\vb,\ \vx_i\in \Xcal_i,\, i=1,2,...,N.
\end{aligned}
\end{equation}
One can easily extend our analysis to the case where $g\neq 0, v\neq 0$ and $g$ is also stochastic. An example of \eqref{eq:main2} is the penalized and constrained regression problem \cite{james2013pcreg} that includes \eqref{prob:classo} as a special case.

Due to the expectation form of $f$, it is natural that the exact gradient of $f$ is not available or very expensive to compute. Instead, we assume that its stochastic gradient is readily accessible. By some slight abuse of the notation, we denote
\begin{equation}\label{w-H-s}
\vw=\left[\begin{array}{c}\vx\\ \vlam\end{array}\right],\quad
H(\vw)=\left[\begin{array}{c}-\vA^\top\vlam\\ \vA\vx-\vb\end{array}\right].
\end{equation}
A point $\vx^*$ is a solution to \eqref{eq:main2} \emph{if and only if} there exists $\vlam^*$ such that
\begin{subequations}\label{1st-opt-s}
\begin{align}
&F(\vx)-F(\vx^*)+(\vw-\vw^*)^\top H(\vw^*)\ge 0,\,\forall \vw,\\
&\vA\vx^*=\vb, \ \vx^*\in \Xcal.
\end{align}
\end{subequations}

Modifying Algorithm \ref{alg:rpdc} to \eqref{eq:main2}, we present the stochastic primal-dual coordinate update method of multipliers, summarized in 
Algorithm \ref{alg:srpdc}, where $G^k$ is a stochastic approximation of $\nabla f(\vx^k)$. The strategy of block coordinate update with stochastic gradient information was first proposed in \cite{dang-lan2015SBMD, xu-yin2015bsg}, which considered problems without linear constraint.

\begin{algorithm}\caption{\textbf{R}andomized \textbf{P}rimal-\textbf{D}ual \textbf{B}lock Coordinate \textbf{U}pdate Method for \textbf{S}tochastic Programming 
(RPDBUS)
}\label{alg:srpdc}
\DontPrintSemicolon
{\small
\textbf{Initialization:} choose $\vx^0, \vlam^0$ and set parameters $\rho, \alpha_k$'s\;
\For{$k=0,1,\ldots$}{
Randomly select $I_k\subset[N]$ with $|I_k|=n$ according to {\bf (U)}.\;
Let $\vx_i^{k+1}=\vx_i^{k},\,\forall i\not\in I_k$, and for $I = I_k$, do the update
\begin{align}\label{eq:supdate-x}
\vx_I^{k+1}=\argmin_{\vx_I\in \Xcal_I}\langle \vG_I^k-\vA_I^\top\vlam^k, \vx_I\rangle + u_I(\vx_I)+\frac{\rho}{2}\|\vA_I(\vx_I-\vx_I^k)+\vr^k\|^2+\frac{1}{2}\|\vx_I-\vx_I^k\|_{\vP^k+\frac{\vI}{\alpha_k}}^2.
\end{align}
Update the residual $\vr^{k+1}=\vr^k+\vA_I(\vx_I^{k+1}-\vx_I^k)$.\;
Update the multiplier by
\begin{equation}
\vlam^{k+1}=\vlam^k-\left(1-\frac{(N-n)\alpha_{k+1}}{N\alpha_k}\right)\rho\vr^{k+1}.\label{eq:supdate-lam}
\end{equation}
}
}
\end{algorithm}
\normalsize
We make the following assumption on the stochastic gradient $\vG^k$.
\begin{assumption}\label{assump-error}
Let $\vdelta^k=\vG^k-\nabla f(\vx^k)$. There exists a constant $\sigma$ such that for all $k$,
\begin{subequations}\label{ass-error1}
\begin{align}
&\EE[\vdelta^k\,|\,\vx^k]=\vzero,\label{ass-error11}\\
&\EE\|\vdelta^k\|^2\le\sigma^2.\label{ass-error12}
\end{align}
\end{subequations}
\end{assumption}

Following the proof of Lemma \ref{lem:1step} and also noting \begin{equation}\label{term6}\EE_{I_k}\big[(\vx_{I_k}-\vx_{I_k}^{k+1})^\top \vdelta_{I_k}^k\,|\,\vx^k\big]=\EE_{I_k}(\vx^k-\vx^{k+1})^\top\vdelta^k,
\end{equation}
we immediately have the following result.
\begin{lemma}[One-step analysis]
Let $\{(\vx^k,\vr^k,\vlam^k)\}$ be the sequence generated from Algorithm \ref{alg:srpdc} where $\vP^k$ is given in \eqref{matPQ} with $\rho_x=\rho$. Then
\begin{eqnarray}\label{ineq-k0-x-s}
& &\EE_{I_k}\big[F(\vx^{k+1})-F(\vx)+(\vx^{k+1}-\vx)^\top(-\vA^\top\vlam^k)+
\rho(\vx^{k+1}-\vx)^\top\vA^\top\vr^{k+1}\big]\cr
& &+\EE_{I_k}(\vx^{k+1}-\vx)^\top\left(\hat{\vP}-\rho A^\top A+\frac{\vI}{\alpha_k}\right)(\vx^{k+1}-\vx^k)
-\frac{L_f}{2}\EE_{I_k}\|\vx^k-\vx^{k+1}\|^2+\EE_{I_k}(\vx^{k+1}-\vx^k)^\top\vdelta^k\cr
&\le & \left(1-\frac{n}{N}\right)\big[F(\vx^k)-F(\vx)+(\vx^{k}-\vx)^\top(-\vA^\top\vlam^k)+
\rho(\vx^{k}-\vx)^\top\vA^\top\vr^{k}\big].
\end{eqnarray}
\end{lemma}


The following theorem is a key result, from which we can choose appropriate $\alpha_k$ to obtain the $O(1/\sqrt{t})$ convergence rate. 

\begin{theorem}\label{thm-s-vx}
Let $\{(\vx^k,\vlam^k)\}$ be the sequence generated from Algorithm \ref{alg:srpdc}.
Let $\theta=\frac{n}{N}$ and denote $$\beta_k=\frac{\alpha_k}{\left(1-\frac{\alpha_{k}(1-\theta)}{\alpha_{k-1}}\right)\rho},\forall k.$$
Assume $\alpha_k>0$ is nonincreasing, and
\begin{subequations}\label{bd-x-s}
\begin{align}
&Ax^0=b, \quad \lambda^0=0,\\
&\hat{\vP}\succeq L_f\vI+\rho A^\top A,\label{bd-x-s-P}\\
&\frac{\alpha_{k-1}\beta_k}{2\alpha_k}+\frac{(1-\theta)\beta_{k+1}}{2}-\frac{\alpha_k\beta_{k+1}}{2\alpha_{k+1}}-\frac{(1-\theta)\beta_k}{2}\ge 0,\,\forall k\label{bd-x-s1}\\
&\frac{\alpha_t}{2\rho}\ge\left|\frac{\alpha_{t-1}\beta_t}{\alpha_t}-(1-\theta)\beta_t-\frac{\alpha_t}{\rho}\right|,\text{ for some }t.\label{bd-x-s2}
\end{align}
\end{subequations}
Let
\begin{equation}\label{avg-pt-3}\hat{\vx}^{t}=\frac{\alpha_{t+1}\vx^{t+1}+\theta\sum_{k=1}^t\alpha_k\vx^k}{\alpha_{t+1}+\theta\sum\limits_{k=1}^t\alpha_k}. \quad 
\end{equation}
Then, under Assumptions~\ref{assump1}, \ref{assump2}, \ref{assump3} and \ref{assump-error}, we have
\begin{eqnarray}\label{rate-x-s}
&&(\alpha_{t+1}+\theta\sum\limits_{k=1}^t\alpha_k)\EE\left[F(\hat{\vx}^{t})-F(\vx^*)+\gamma\|A\hat{x}^t-b\|\right]\cr
&\le & (1-\theta)\alpha_0\left[F(\vx^0)-F(\vx^*)\right]+\frac{\alpha_0}{2}\|\vx^{0}-\vx^*\|_{\hat{\vP}-\rho A^\top A}^2+\frac{1}{2}\|\vx^0-\vx^*\|^2\cr
&&+\left|\frac{\alpha_0\beta_1}{2\alpha_1}-\frac{(1-\theta)\beta_1}{2}\right|\gamma^2+\sum_{k=0}^t\frac{\alpha_k^2}{2}\EE\|\vdelta^k\|^2.
\end{eqnarray}
\end{theorem}

The following proposition gives sublinear convergence rate of Algorithm \ref{alg:srpdc} by specifying the values of its parameters. 
The choice of $\alpha_k$ depends on whether we fix the total number of iterations.

\begin{proposition}\label{prop:rate-s}
Let $\{(\vx^k,\vlam^k)\}$ be the sequence generated from Algorithm \ref{alg:srpdc} with $\vP^k$ given in \eqref{matPQ}, $\hat{\vP}$ satisfying \eqref{bd-x-s-P}, and the initial point satisfying $\vA\vx^0=\vb$ and $\vlam^0=\vzero$. Let $C_0$ be
\begin{align}\label{eq:rate-s-c}
C_0=(1-\theta)\alpha_0\big[F(x^0)-F(x^*)\big]+\frac{1}{2}\|x^0-x^*\|_{D_x}^2+\frac{\alpha_0}{2\rho}\max\{(1+\|\lambda^*\|)^2, 4\|\lambda^*\|^2\},
\end{align}
where $(x^*,\lambda^*)$ is a primal-dual solution, and
$\vD_x:=\alpha_0(\hat{\vP}-\rho A^\top A)+\vI$.

\begin{enumerate}
\item If $\alpha_k=\frac{\alpha_0}{\sqrt{k}},\forall k\ge1$ for a certain $\alpha_0>0$, then for $t\ge 2$,
{ 
\begin{equation}\label{rate-sqrtk}
\max\left\{\left|\EE[F(\hat{\vx}^{t})-F(\vx^*)]\right|,\, \EE\|A\hat{x}^t-b\|\right\}\le \frac{C_0}{\theta\alpha_0\sqrt{t}}+\frac{\alpha_0(\log t+2)\sigma^2}{2\theta\sqrt{t}}.
\end{equation}
}

\item If the number of maximum number of iteration is fixed a priori, 
then by choosing $\alpha_k=\frac{\alpha_0}{\sqrt{t}},\,\forall k\ge 1$ with any given $\alpha_0>0$, we have
{ 
\begin{equation}\label{rate-sqrtk-2}
\max\left\{\left|\EE[F(\hat{\vx}^{t})-F(\vx^*)]\right|,\, \EE\|A\hat{x}^t-b\|\right\}\le\frac{C_0}{\theta\alpha_0\sqrt{t}}+\frac{\alpha_0\sigma^2}{\theta\sqrt{t}}.
\end{equation}
}
\end{enumerate}
\end{proposition}

\begin{proof}
When $\alpha_k=\frac{\alpha_0}{\sqrt{k}}$, we can show that \eqref{bd-x-s1} and \eqref{bd-x-s2} hold for $t\ge 2$; see Appendix \ref{app:bd-x-s}. Hence, the result in \eqref{rate-sqrtk} follows from \eqref{rate-x-s}, the convexity of $F$, Lemma \ref{equiv-rate} with $\gamma=\max\{1+\|\lambda^*\|, 2\|\lambda^*\|\}$, and the inequalities
$$\sum_{k=1}^t\frac{1}{\sqrt{k}}\ge \sqrt{t},\quad \sum_{k=1}^t\frac{1}{k}\le \log t +1.$$
When $\alpha_k$ is a constant, the terms on the left hand side of \eqref{bd-x-s1} and on the right hand side of \eqref{bd-x-s2} are both zero, so they are satisfied. Hence, the result in \eqref{rate-sqrtk-2} immediately follows by noting $\sum_{k=1}^t \alpha_k=\alpha_0\sqrt{t}$ and $\sum_{k=0}^t \alpha_k^2\le 2\alpha_0^2$.
\hfill\end{proof}

The sublinear convergence result of Algorithm \ref{alg:srpdc} can also be shown if $f$ is nondifferentiable convex and Lipschitz continuous. Indeed, if $f$ is Lipschtiz continuous with constant $L_c$, i.e., $$\|f(\vx)-f(\vy)\|\le L_c\|\vx-\vy\|,\,\forall \vx, \vy,$$
then $\|\tilde{\nabla} f(\vx)\|\le L_c,\,\forall \vx$, where $\tilde{\nabla}f(\vx)$ is a subgradient of $f$ at $\vx$. Hence,
\begin{eqnarray*}
& &\EE_{I_k}(\vx_{I_k}-\vx_{I_k}^{k+1})^\top\tilde{\nabla}_{I_k}f(\vx^k)\cr
&=& \EE_{I_k}(\vx_{I_k}-\vx_{I_k}^k)^\top\tilde{\nabla}_{I_k}f(\vx^k) + \EE_{I_k}(\vx_{I_k}^k-\vx_{I_k}^{k+1})^\top\tilde{\nabla}_{I_k}f(\vx^k)\cr
&=&\frac{n}{N}(\vx-\vx^k)^\top \tilde{\nabla} f(\vx^k) + \EE_{I_k}(\vx^k-\vx^{k+1})^\top\tilde{\nabla} f(\vx^{k+1})+\EE_{I_k}(\vx^k-\vx^{k+1})^\top\big(\tilde{\nabla}f(\vx^k)-\tilde{\nabla} f(\vx^{k+1})\big)\\
&\le & \frac{n}{N}(f(\vx)-f(\vx^k)) + \EE_{I_k}[f(\vx^k)-f(\vx^{k+1})]+\EE_{I_k}(\vx^k-\vx^{k+1})^\top\big(\tilde{\nabla}f(\vx^k)-\tilde{\nabla} f(\vx^{k+1})\big)\cr
&=&\frac{n-N}{N}(f(\vx)-f(\vx^k))+\EE_{I_k}[f(\vx)-f(\vx^{k+1})]+\EE_{I_k}(\vx^k-\vx^{k+1})^\top\big(\tilde{\nabla}f(\vx^k)-\tilde{\nabla} f(\vx^{k+1})\big).
\end{eqnarray*}
Now following the proof of Lemma \ref{lem:1step}, we can have a result similar to \eqref{ineq-k0-x-s}, and then through the same arguments as those in the proof of Theorem \ref{thm-s-vx}, we can establish sublinear convergence rate of $O(1/\sqrt{t})$.

{ 
\section{Numerical Experiments}\label{sec:numerical}
In this section, we test the proposed randomized primal-dual method on solving the nonnegativity constrained quadratic programming (NCQP):
\begin{equation}\label{eq:ncqp}
\min_{x\in\RR^n} F(x)\equiv\frac{1}{2}x^\top Q x + c^\top x, \st Ax=b, x_i\ge 0, i=1,\ldots, n,
\end{equation}
where $A\in\RR^{m\times n}$, and $Q\in\RR^{n\times n}$ is a symmetric positive semidefinite (PSD) matrix. There is no $y$-variable, and it falls into the case in Theorem \ref{thm:rate-3X}. We perform two experiments on a Macbook Pro with 4 cores. The first experiment demonstrates the parallelization performance of the proposed method, and the second one compares it to other methods. 


\textbf{Parallelization.} This test is to illustrate the power unleashed in our new method, which is flexible in terms of parallel and distributive computing. We set $m = 200, n=2000$ and generate $Q=HH^\top$, where the components of $H\in \RR^{n\times n}$ follow the standard Gaussian distribution. The matrix $A$ and vectors $b, c$ are also randomly generated. We treat every component of $x$ as one block, and at every iteration we select and update $p$ blocks, where $p$ is the number of used cores. Figure \ref{parallel-comp} shows the running time by using 1, 2, and 4 cores, where the optimal value $F(x^*)$ is obtained by calling Matlab function \verb|quadprog| with tolerance $10^{-16}$. From the figure, we see that our proposed method achieves nearly linear speed-up.

\begin{figure}
\centering \subfigure{\includegraphics[scale=0.45]{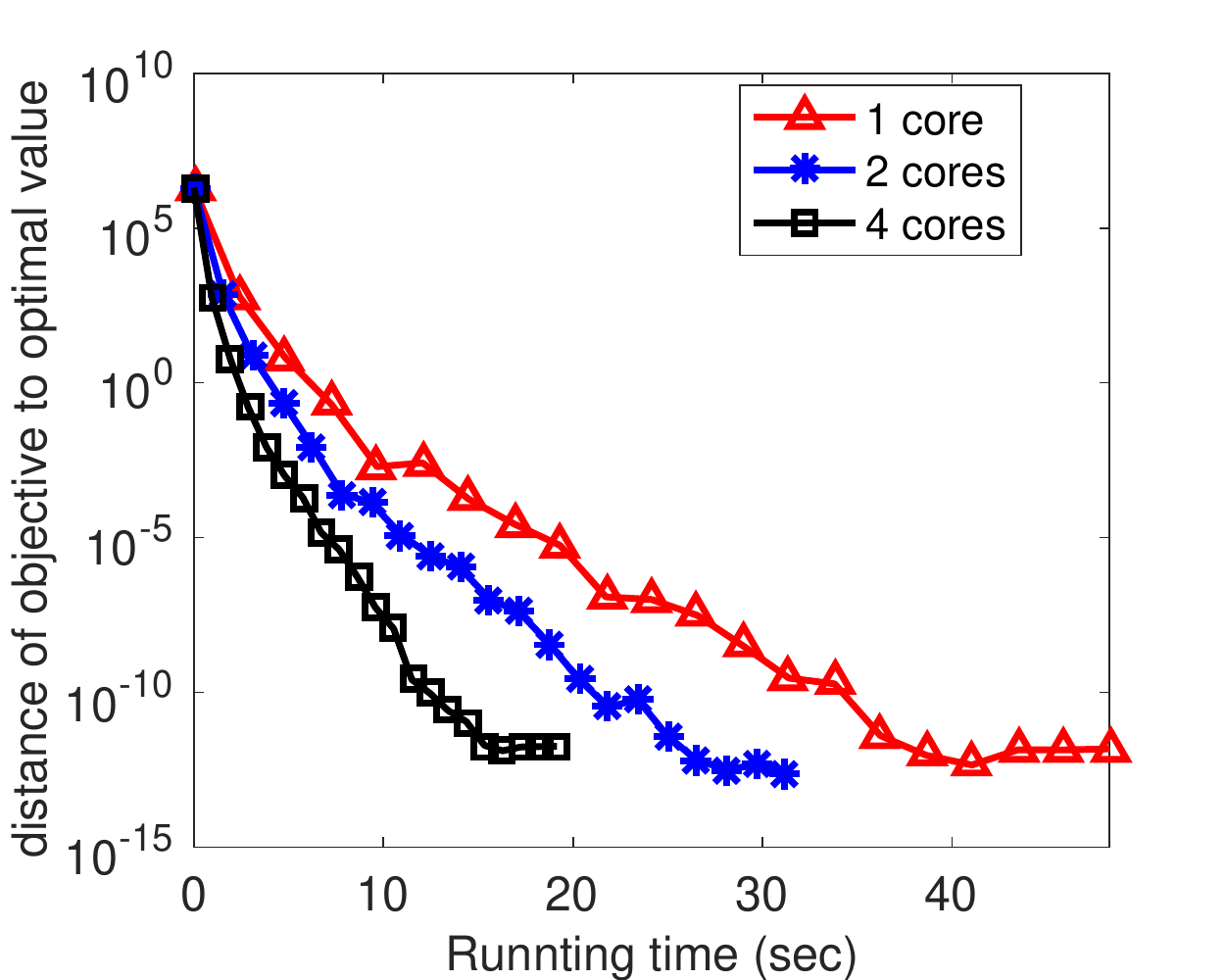}}
\centering \subfigure{\includegraphics[scale=0.45]{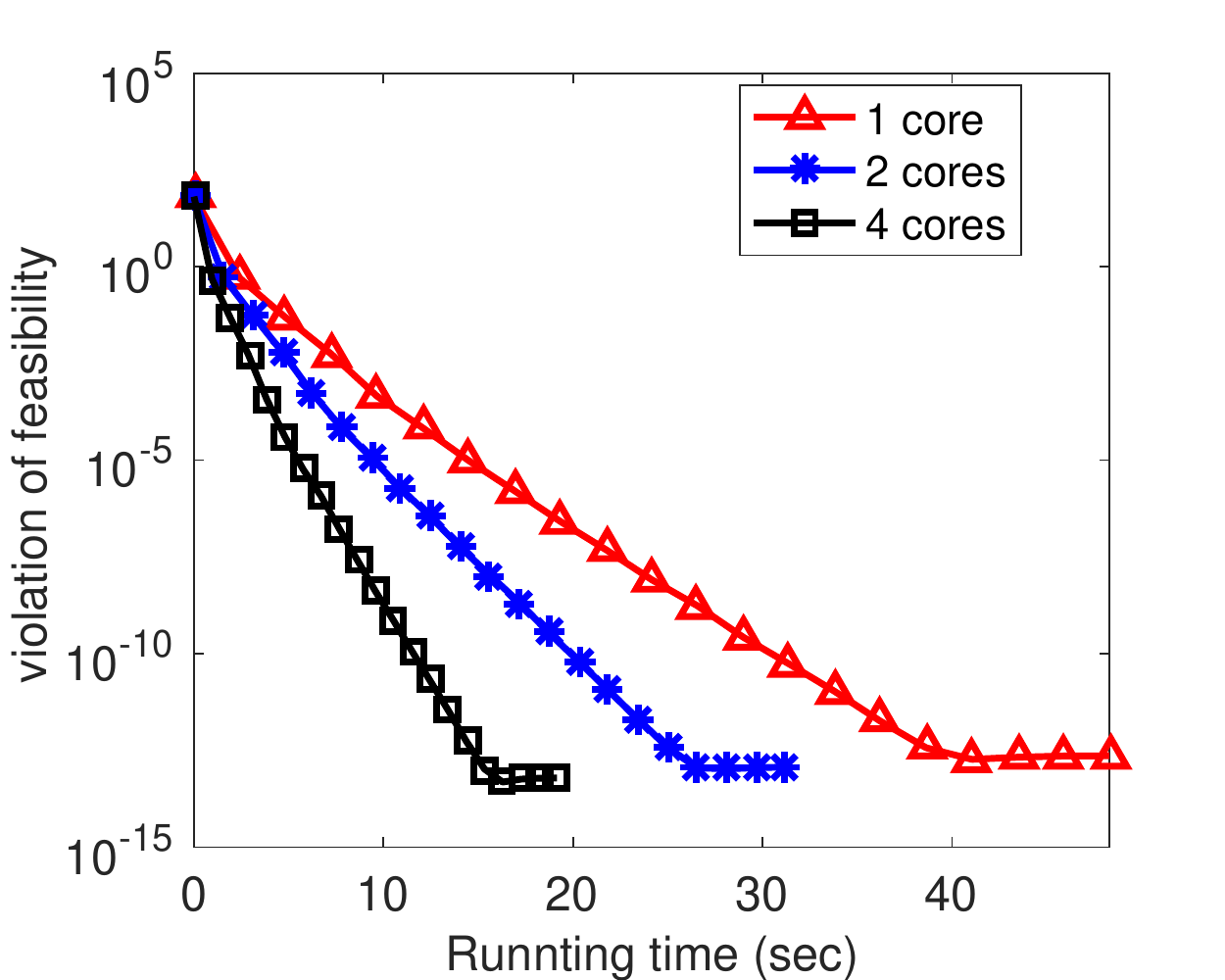}}
\caption{Nearly linear speed-up performance of the proposed primal-dual method for solving \eqref{eq:ncqp} on a 4-core machine. Left: distance of objective to optimal value $|F(x^k)-F(x^*)|$; Right: violation of feasibility $\|Ax^k-b\|$.}
\label{parallel-comp}
\end{figure}

{\bf Comparison to other methods.} In this experiment, we compare the proposed method to the linearized ALM and the cyclic linearized ADMM methods. We set $m = 1000, n=5000$ and generate $Q=HH^\top$, where the components of $H\in\RR^{n\times (n-50)}$ follow standard Gaussian distribution. Note that $Q$ is singular, and thus \eqref{eq:ncqp} is not strongly convex. We partition the variable into 100 blocks, each with 50 components. At each iteration of our method, we randomly select one block variable to update. Figure \ref{random-QP} shows the performance by the three compared methods, where one epoch is equivalent to updating 100 blocks once. From the figure, we see that our proposed method is comparable to the cyclic linearized ADMM and significantly better than the linearized ALM. Although the cyclic ADMM performs well on this example, in general it can diverge if the problem has more than two blocks; see \cite{chen2013direct}.

\begin{figure}
\centering \subfigure{\includegraphics[scale=0.45]{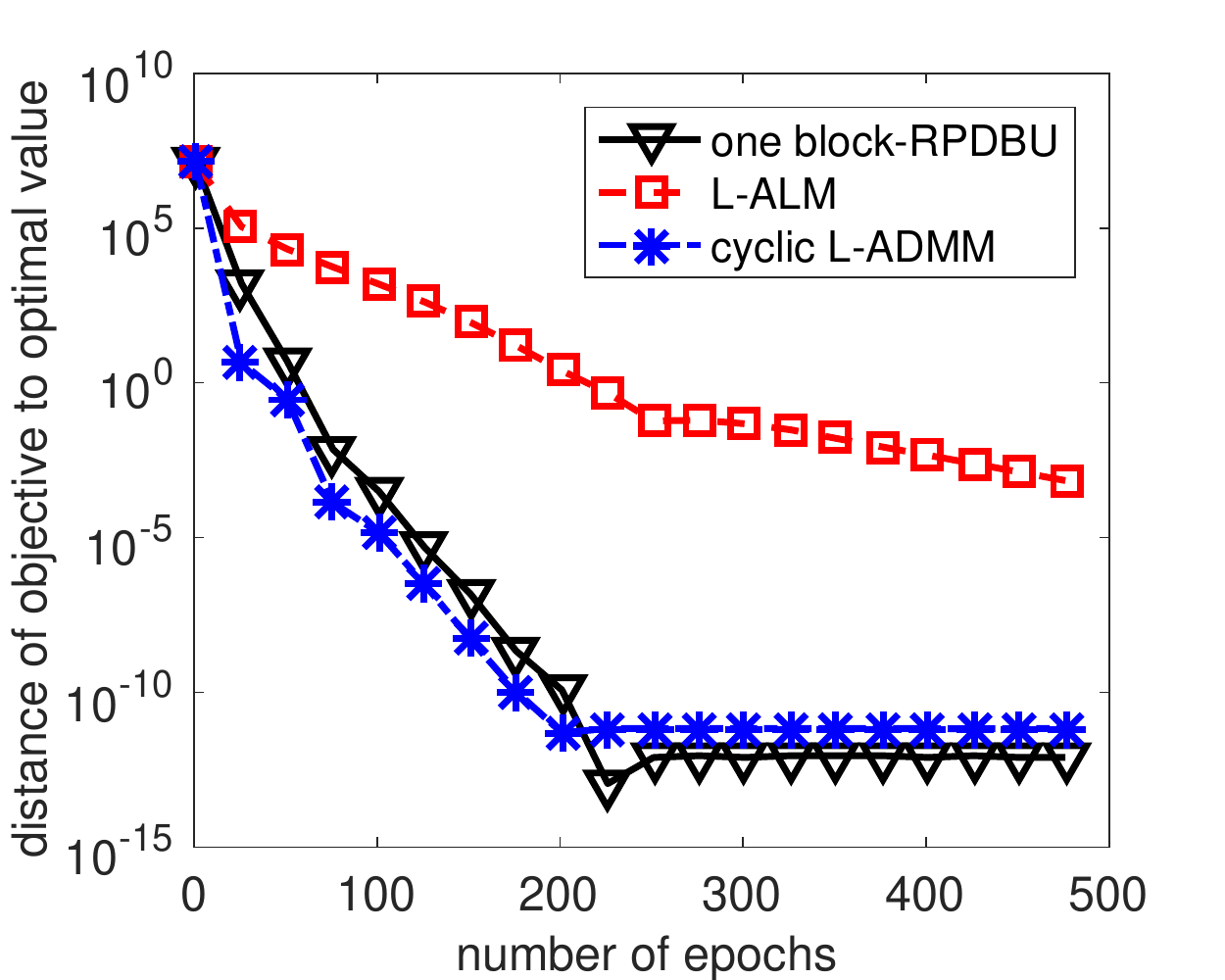}}
\centering \subfigure{\includegraphics[scale=0.45]{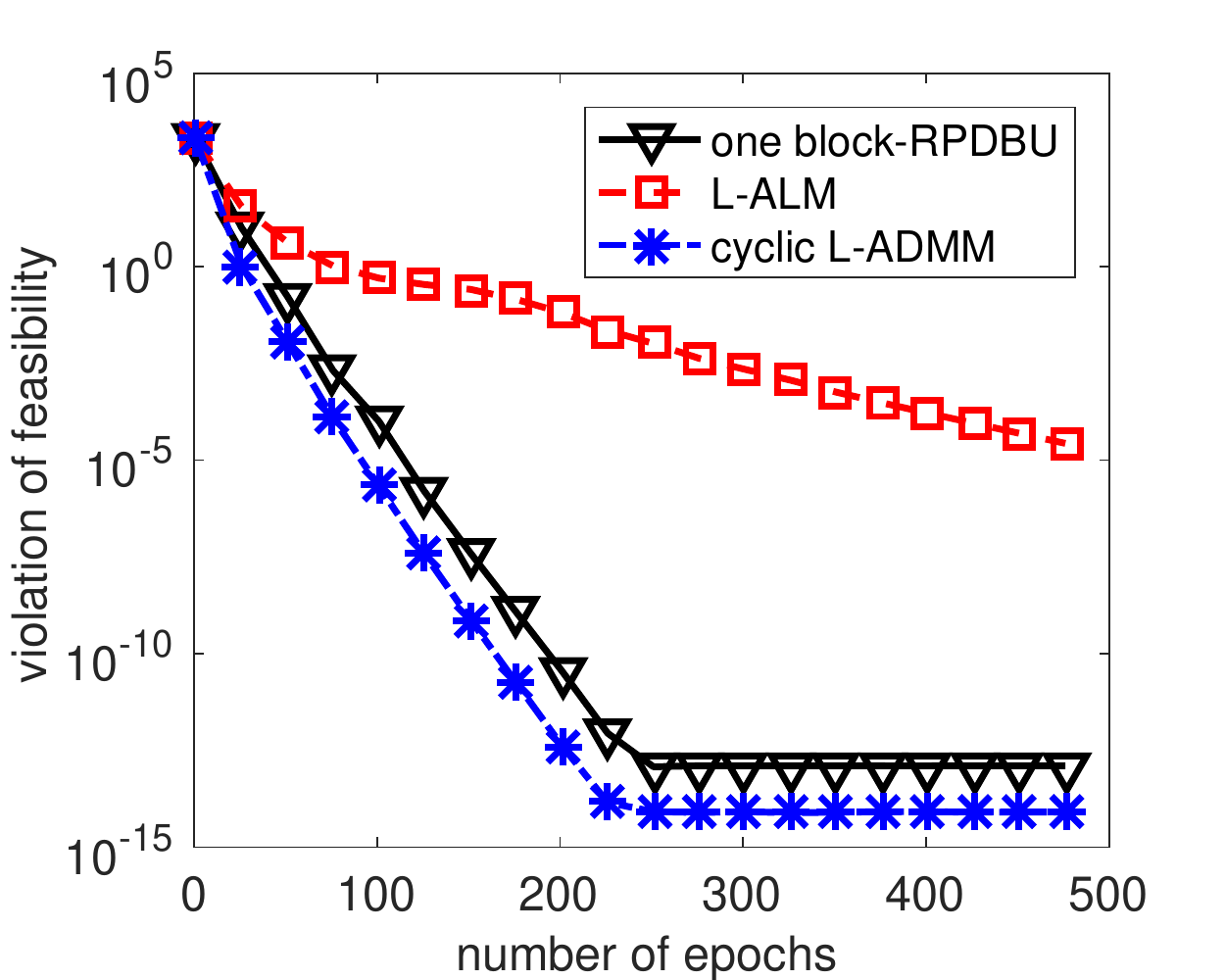}}
\caption{Comparison of the proposed method (RPDBU) to the linearized augmented Lagrangian method (L-ALM) and the cyclic linearized alternating direction method of multipliers (L-ADMM) on solving the nonnegativity constrained quadratic programming \eqref{eq:ncqp}. Left: distance of objective to optimal value $|F(x^k)-F(x^*)|$; Right: violation of feasibility $\|Ax^k-b\|$.}
\label{random-QP}
\end{figure}

}

\section{Connections to Existing Methods}\label{sec:connection}
In this section, we discuss how Algorithms \ref{alg:rpdc} and \ref{alg:srpdc} are related to several existing methods in the literature, and we also compare their convergence results. It turns out that the proposed algorithms specialize to several known methods or their variants in the literature under various specific conditions. Therefore, our convergence analysis recovers some existing results as special cases, as well as provides new convergence results for certain existing algorithms such as the Jacobian proximal parallel ADMM and the primal-dual scheme in \eqref{alg:r1st-pd}.

\subsection{Randomized proximal coordinate descent}
The randomized proximal coordinate descent (RPCD) was proposed in \cite{nesterov2012rcd}, where smooth convex optimization problems are considered. It was then extended in \cite{richtarik2014iteration, lu2015complexity} to deal with nonsmooth problems that can be formulated as
\begin{equation}\label{prob:rbcd}
\min_x f(x_1,\cdots,x_N)+\sum_{i=1}^N u_i(x_i),
\end{equation}
where $x=(x_1;\ldots;x_N)$. Toward solving \eqref{prob:rbcd}, at each iteration $k$, the RPCD method first randomly selects one block $i_k$ and then performs the update:
\begin{equation}\label{alg:rbcd}
x_i^{k+1}=\left\{
\begin{array}{ll}
\argmin_{x_i} \langle \nabla_i f(x^k), x_i\rangle + \frac{L_i}{2}\|x_i-x_i^k\|_2^2+u_i(x_i), & \text{ if }i=i_k,\\
x_i^k, & \text{ if } i\neq i_k,
\end{array}\right.
\end{equation}
where $L_i$ is the Lipschitz continuity constant of the partial gradient $\nabla_i f(x)$. With more than one blocks selected every time, \eqref{alg:rbcd} has been further extended into parallel coordinate descent in \cite{richtarik2015parallel}.

When there is no linear constraint and no $y$-variable in \eqref{eq:main}, then Algorithm \ref{alg:rpdc} reduces to the scheme in \eqref{alg:rbcd} if $I_k=\{i_k\}$, i.e., only one block is chosen, and $P^k= L_{i_k}I, \lambda^k=0,\,\forall k$, and to the parallel coordinate descent in \cite{richtarik2015parallel} if $I_k=\{i_k^1,\cdots,i_k^n\}$ and $P^k=\text{blkdiag}(L_{i_k^1}I,\cdots, L_{i_k^n}I), \lambda^k=0,\,\forall k$. Although the convergence rate results in \cite{richtarik2014iteration, lu2015complexity, richtarik2015parallel} are non-ergodic, we can easily strengthen our result to a non-ergodic one by noticing that \eqref{eq:1step-x} implies nonincreasing monotonicity of the objective if Algorithm \ref{alg:rpdc} is applied to \eqref{prob:rbcd}.

\subsection{Stochastic block proximal gradient}
For solving the problem \eqref{prob:rbcd} with a stochastic $f$, \cite{dang-lan2015SBMD} proposes a stochastic block proximal gradient (SBPG) method, which iteratively performs the update in \eqref{alg:rbcd} with $\nabla_i f(x^k)$ replaced by a stochastic approximation. If $f$ is Lipschitz differentiable, then an ergodic $O(1/\sqrt{t})$ convergence rate was shown. Setting $I_k=\{i_k\},\forall k$, we reduce Algorithm \ref{alg:srpdc} to the SBPG method, and thus our convergence results in Proposition \ref{prop:rate-s} recover that in \cite{dang-lan2015SBMD}.

\subsection{Multi-block ADMM}
Without coupled functions or proximal terms, Algorithm \ref{alg:rpdc} can be regarded as a randomized variant of the multi-block ADMM scheme in \eqref{alg:multi-adm}. While multi-block ADMM can diverge if the problem has three or more blocks, our result in Theorem \ref{thm:rate-3X} shows that $O(1/t)$ convergence rate is guaranteed if at each iteration, one randomly selected block is updated, followed by an update to the multiplier. Note that in the case of no coupled function and $n=1$, \eqref{para-mat-3X} indicates that we can choose $P^k=0$, i.e.\ without proximal term. Hence, randomization is a key to convergence.

When there are only two blocks, ADMM has been shown (e.g., \cite{LMZ2015JORSC}) to have an ergodic $O(1/t)$ convergence rate. If there are no coupled functions, \eqref{para-mat-1Yw} and \eqref{para-mat} both indicate that we can choose $\hat{P}=\rho_x A^\top A, \hat{Q}=\rho_y B^\top B$ if $\theta=1$, i.e., all $x$ and $y$ blocks are selected. Thus according to \eqref{matPQ}, we can set $P^k=0, Q^k=0, \,\forall k$, in which case Algorithm \ref{alg:rpdc} reduces to the classic 2-block ADMM. Hence, our results in Theorems \ref{thm:rate-1Yw} and \ref{thm:rate-cvx}  both recover the ergodic $O(1/t)$ convergence rate of ADMM for two-block convex optimization problems.

\subsection{Proximal Jacobian parallel ADMM}
In \cite{deng2013parallel}, the proximal Jacobian parallel ADMM (Prox-JADMM) was proposed to solve the linearly constrained multi-block separable convex optimization model
\begin{equation}\label{prob:jpadmm}
\min_x \sum_{i=1}^N u_i(x_i), \st \sum_{i=1}^N A_i x_i=b.
\end{equation}
At each iteration, the Prox-JADMM method performs the updates for $i=1,\ldots,n$ in parallel:
\begin{equation}\label{alg:jpadmm-x}
x_i^{k+1}=\argmin_{x_i} u_i(x_i)-\langle\lambda^k, A_ix_i\rangle+\frac{\rho}{2}\big\|A_ix_i+\sum_{j\neq i}A_jx_j^k-b\big\|_2^2+\frac{1}{2}\|x_i-x_i^k\|_{P_i}^2,
\end{equation}
and then updates the multiplier by
\begin{equation}\label{alg:jpadmm-lam}
\lambda^{k+1}=\lambda^k-\gamma\rho\left(\sum_{i=1}^N A_ix_i^{k+1}-b\right),
\end{equation}
where $P_i\succ 0,\forall i$ and $\gamma>0$ is a damping parameter. By choosing approapriate parameters, \cite{deng2013parallel} established convergence rate of order $1/t$ based on norm square of the difference of two consecutive iterates.

If there is no $y$-variable or the coupled function $f$ in \eqref{eq:main}, setting $I_k=[N], P^k=\mathrm{blkdiag}(\rho_x A_1^\top A_1+P_1,\cdots,\rho_x A_N^\top A_N+P_N)-\rho_x A^\top A\succeq 0,\,\forall k$, where $P_i$'s are the same as those in \eqref{alg:jpadmm-x}, then Algorithm~\ref{alg:rpdc} reduces to the Prox-JADMM with $\gamma=1$, and Theorem~\ref{thm:rate-3X} provides a new convergence result in terms of the objective value and the feasibility measure.

\subsection{Randomized primal-dual scheme in \eqref{alg:r1st-pd}}
In this subsection, we show that the scheme in \eqref{alg:r1st-pd} is a special case of Algorithm \ref{alg:rpdc}. Let $g$ be the convex conjugate of $g^*:=h+\iota_{\cZ}$, namely, $g(y)=\sup_{z} \langle y, z\rangle-h(z)-\iota_\cZ(z)$. Then \eqref{eq:saddle-prob} is equivalent to the optimization problem:
$$\min_{x\in\cX} \sum_{i=1}^N u_i(x_i)+g(-Ax),$$
which can be further written as
\begin{equation}\label{equiv-prob}
\min_{x\in\cX, y} \sum_{i=1}^N u_i(x_i)+g(y), \st Ax+y=0.
\end{equation}
\begin{proposition}\label{prop-equiv-pd}
The scheme in \eqref{alg:r1st-pd} is equivalent to the following updates:
\begin{subequations}\label{equiv-alg}
\begin{align}
&x_i^{k+1}=\left\{
\begin{array}{ll}
\underset{x_i\in \cX_i}\argmin\langle -z^k, A_i x_i\rangle +u_i(x_i)+\frac{q}{2\eta}\|A_i (x_i-x_i^k)+r^k\|^2+\frac{1}{2}\|x_i-x_i^k\|_{\tau I-\frac{q}{\eta}A_i^\top A_i}, &~i=i_k,\\
x_i^k, &~i\neq i_k,
\end{array}\right. \label{equiv-x}\\
& y^{k+1}=\argmin_y g(y)-\langle y, z^k\rangle+\frac{1}{2\eta}\| y+ A x^{k+1}\|^2,\label{equiv-y}\\
&z^{k+1}=z^k-\frac{1}{\eta}(Ax^{k+1}+y^{k+1}), \label{equiv-z}
\end{align}
\end{subequations}
where $r^k=Ax^k+y^k$. Therefore, it is a special case of Algorithm \ref{alg:rpdc} applied to \eqref{equiv-prob} with the setting of $I_k=\{i_k\}, \rho_x=\frac{q}{\eta}, \rho_y=\rho=\frac{1}{\eta}$ and $P^k=\tau I-\frac{q}{\eta}A_{i_k}^\top A_{i_k}, Q^k=0,\forall k$.
\end{proposition}
While the sublinear convergence rate result in \cite{dang2014randomized} requires the boundedness of $\cX$ and $\cZ$, the result in Theorem \ref{thm:rate-1Yw} indicates that the boundedness assumption can be removed if we add one proximal term to the $y$-update in \eqref{equiv-y}.

\section{Concluding Remarks}\label{sec:conc-rem}
We have proposed a randomized primal-dual coordinate update algorithm, called RPDBU, 
for solving linearly constrained convex optimization with multi-block decision variables and coupled terms in the objective. By using a randomization scheme and the proximal gradient mappings, we show a sublinear convergence rate of the RPDBU 
method.
In particular, without any assumptions other than convexity on the objective function and without imposing any restrictions on the constraint matrices, an $O(1/t)$ convergence rate is established. 
We have also extended RPDBU 
to solve the problem where the objective is stochastic. If a stochastic (sub-)gradient estimator is available, then we show that by adaptively choosing the parameter $\alpha_k$ in the added proximal term, an $O(1/\sqrt{t})$ convergence rate can be established. 
Furthermore, if there is no coupled function $f$, then we can remove the proximal term, and the algorithm reduces to a randomized multi-block ADMM. Hence, the convergence of the original randomized multi-block ADMM follows as a consequence of our analysis. Remark also that by taking the sampling set $I_k$ as the whole set and $P^k$ as some special matrices, our algorithm specializes to the proximal Jacobian ADMM.
Finally, we pose as an open problem to decide whether or not
a deterministic counterpart of the RPDBU 
exists, retaining similar convergence properties for solving problem \eqref{eq:main}.
For instance, it would be interesting to know if the algorithm would still be convergent if a deterministic cyclic update rule is applied while a proper proximal term is incorporated.


\bibliographystyle{abbrv}

{\small
\appendix
\section{Proofs of Lemmas}\label{sec:app-A}
We give proofs of several lemmas that are used to show our main results. Throughout our proofs, we define $\tilde{P}$ and $\tilde{Q}$ as follows:
\begin{equation}\label{matPQtilde}
\tilde{P}=\hat{P}-\rho_x A^\top A,\qquad \tilde{Q}=\hat{Q}-\rho_y B^\top B.
\end{equation}
\subsection{Proof of Lemma \ref{lem:1step}}
We prove \eqref{ineq-k1-x}, and \eqref{ineq-k1-y} can be shown by the same arguments.
By the optimality of $\vx_{I_k}^{k+1}$, we have for any $\vx_{I_k}\in \cX_{I_k}$,
\begin{equation}\label{eq:optjk}
(\vx_{I_k}-\vx_{I_k}^{k+1})^\top\left(\nabla_{I_k}f(\vx^k)-\vA_{I_k}^\top\vlam^k
+\tilde{\nabla} u_{I_k}(\vx_{I_k}^{k+1})+\rho_x\vA_{I_k}^\top\vr^{k+\frac{1}{2}}
+\vP^k(\vx_{I_k}^{k+1}-\vx_{I_k}^k)\right)\ge 0,
\end{equation}
where $\tilde{\nabla} u_{I_k}(\vx_{I_k}^{k+1})$ is a subgradient of $u_{I_k}$ at $\vx_{I_k}^{k+1}$, and we have used the formula of $\vr^{k+\frac{1}{2}}$ given in \eqref{eq:update-r1}.
We compute the expectation of each term in \eqref{eq:optjk} in the following. First, we have
\begin{eqnarray}
& &\EE_{I_k}(\vx_{I_k}-\vx_{I_k}^{k+1})^\top\nabla_{I_k}f(\vx^k)\cr
&=& \EE_{I_k}\left(\vx_{I_k}-\vx_{I_k}^k\right)^\top\nabla_{I_k}f(\vx^k) + \EE_{I_k}(\vx_{I_k}^k-\vx_{I_k}^{k+1})^\top\nabla_{I_k}f(\vx^k)\cr
&=&\frac{n}{N}\left(\vx-\vx^k\right)^\top \nabla f(\vx^k) + \EE_{I_k}(\vx^k-\vx^{k+1})^\top\nabla f(\vx^k)\label{gradf-term}\\
&\le & \frac{n}{N}\left(f(\vx)-f(\vx^k)\right) + \EE_{I_k}\left[f(\vx^k)-f(\vx^{k+1})+\frac{L_f}{2}\|\vx^k-\vx^{k+1}\|^2\right]\cr
&=&\frac{n-N}{N}\left(f(\vx)-f(\vx^k)\right)+\EE_{I_k}\left[f(\vx)-f(\vx^{k+1})+\frac{L_f}{2}\|\vx^k-\vx^{k+1}\|^2\right],\label{term1}
\end{eqnarray}
where the last inequality is from the convexity of $f$ and the Lipschitz continuity of $\nabla_{I_k} f(\vx)$. Secondly,
\begin{eqnarray}\label{term3}
& &\EE_{I_k}(\vx_{I_k}-\vx_{I_k}^{k+1})^\top(-\vA_{I_k}^\top\vlam^k)\cr
&=&\EE_{I_k}(\vx_{I_k}-\vx_{I_k}^{k})^\top(-\vA_{I_k}^\top\vlam^k)+\EE_{I_k}(\vx_{I_k}^k-\vx_{I_k}^{k+1})^\top(-\vA_{I_k}^\top\vlam^k)\cr
&=&\frac{n}{N}\left(\vx-\vx^k\right)^\top(-\vA^\top\vlam^k)+\EE_{I_k}(\vx^k-\vx^{k+1})^\top(-\vA^\top\vlam^k)\cr
&=&\frac{n-N}{N}\left(\vx-\vx^k\right)^\top(-\vA^\top\vlam^k)+\EE_{I_k}(\vx-\vx^{k+1})^\top(-\vA^\top\vlam^k).
\end{eqnarray}
For the third term of \eqref{eq:optjk}, we have
\begin{eqnarray}\label{term2}
& &\EE_{I_k}(\vx_{I_k}-\vx_{I_k}^{k+1})^\top\tilde{\nabla} u_{I_k}(\vx_{I_k}^{k+1})\cr
&\le &\EE_{I_k}\left[u_{I_k}(\vx_{I_k})-u_{I_k}(\vx_{I_k}^{k+1})\right]\cr
&=& \frac{n}{N}\, u(\vx)-\EE_{I_k}[u(\vx^{k+1})-u(\vx^k)+u_{I_k}(\vx_{I_k}^k)]\cr
&=&\frac{n}{N}\left[u(\vx)-u(\vx^k)\right]+\EE_{I_k}[u(\vx^k)-u(\vx^{k+1})]\cr
&=&\frac{n-N}{N}\left[u(\vx)-u(\vx^k)\right]+\EE_{I_k}[u(\vx)-u(\vx^{k+1})],
\end{eqnarray}
where the inequality is from the convexity of $u_{I_k}$.
The expectation of the fourth term of \eqref{eq:optjk} is
\begin{eqnarray}\label{term4}
& & \rho_x \EE_{I_k}(\vx_{I_k}-\vx_{I_k}^{k+1})^\top\vA_{I_k}^\top\vr^{k+\frac{1}{2}}\cr
&=& \rho_x \EE_{I_k}(\vx_{I_k}-\vx_{I_k}^{k+1})^\top\vA_{I_k}^\top\vr^{k}+\rho_x\EE_{I_k}(\vx_{I_k}-\vx_{I_k}^{k+1})^\top\vA_{I_k}^\top\vA_{I_k}(\vx_{I_k}^{k+1}-\vx_{I_k}^k)\cr
&=&\rho_x \EE_{I_k}(\vx_{I_k}-\vx_{I_k}^{k})^\top\vA_{I_k}^\top\vr^{k}+\rho_x\EE_{I_k}(\vx_{I_k}^k
-\vx_{I_k}^{k+1})^\top\vA_{I_k}^\top\vr^{k}+\rho_x\EE_{I_k}(\vx_{I_k}-\vx_{I_k}^{k+1})^\top\vA_{I_k}^\top\vA_{I_k}(\vx_{I_k}^{k+1}-\vx_{I_k}^k)\cr
&=&\frac{n\rho_x}{N}(\vx-\vx^k)^\top\vA^\top\vr^k+\rho_x\EE_{I_k}(\vx^k-\vx^{k+1})^\top\vA^\top\vr^{k}
+\rho_x\EE_{I_k}(\vx_{I_k}-\vx_{I_k}^{k+1})^\top\vA_{I_k}^\top\vA_{I_k}(\vx_{I_k}^{k+1}-\vx_{I_k}^k)\cr
&=&\frac{n-N}{N}\rho_x(\vx-\vx^k)^\top\vA^\top\vr^k+\rho_x\EE_{I_k}(\vx-\vx^{k+1})^\top\vA^\top\vr^{k}
+\rho_x\EE_{I_k}(\vx_{I_k}-\vx_{I_k}^{k+1})^\top\vA_{I_k}^\top\vA_{I_k}(\vx_{I_k}^{k+1}-\vx_{I_k}^k). \nonumber \\
\end{eqnarray}
Finally, we have
\begin{eqnarray}\label{term5}
& &\EE_{I_k}(\vx_{I_k}-\vx_{I_k}^{k+1})^\top\vP^k(\vx_{I_k}^{k+1}-\vx_{I_k}^k)\cr
&=&\EE_{I_k}(\vx-\vx^{k+1})^\top\hat{\vP}(\vx^{k+1}-\vx^k)-\rho_x\EE_{I_k}(\vx_{I_k}-\vx_{I_k}^{k+1})^\top\vA_{I_k}^\top\vA_{I_k}(\vx_{I_k}^{k+1}-\vx_{I_k}^k)\cr
&=&\EE_{I_k}(\vx-\vx^{k+1})^\top\tilde{\vP}(\vx^{k+1}-\vx^k)-\rho_x\EE_{I_k}(\vx-\vx^{k+1})^\top\vA^\top\vA(\vx^k-\vx^{k+1})\\
& &-\rho_x\EE_{I_k}(\vx_{I_k}-\vx_{I_k}^{k+1})^\top\vA_{I_k}^\top\vA_{I_k}(\vx_{I_k}^{k+1}-\vx_{I_k}^k),\nonumber
\end{eqnarray}
where we used the formulas of $\vP$ in \eqref{matPQ} and \eqref{matPQtilde}.

Plugging \eqref{term1} through \eqref{term5} into \eqref{eq:optjk} and recalling $F(\vx)=f(\vx)+u(\vx)$, by rearranging terms we have
\begin{eqnarray}\label{ineq-k0-x}
& &\EE_{I_k}\left[F(\vx^{k+1})-F(\vx)+(\vx^{k+1}-\vx)^\top(-\vA^\top\vlam^k)+
\rho_x(\vx^{k+1}-\vx)^\top\vA^\top\vr^{k+\frac{1}{2}}\right]\cr
& &+\EE_{I_k}\left(\vx^{k+1}-\vx\right)^\top\tilde{\vP}(\vx^{k+1}-\vx^k)-\frac{L_f}{2}\EE_{I_k}\|\vx^k-\vx^{k+1}\|^2\cr
& \le & \frac{N-n}{N}\left[F(\vx^k)-F(\vx)+(\vx^{k}-\vx)^\top(-\vA^\top\vlam^k)+
\rho_x(\vx^{k}-\vx)^\top\vA^\top\vr^{k}\right].
\end{eqnarray}
Note
\begin{eqnarray}
& &(\vx^{k+1}-\vx)^\top(-\vA^\top\vlam^k)+
\rho_x(\vx^{k+1}-\vx)^\top\vA^\top\vr^{k+\frac{1}{2}}\cr
&=&(\vx^{k+1}-\vx)^\top(-\vA^\top\vlam^k)+\rho_x(\vx^{k+1}-\vx)^\top\vA^\top\vr^{k+1}
-\rho_x(\vx^{k+1}-\vx)^\top\vA^\top\vB(\vy^{k+1}-\vy^k)\label{r-half-to-one}\\
&{\overset{\eqref{eq:update-lam}}=}&(\vx^{k+1}-\vx)^\top(-\vA^\top\vlam^{k+1})+(\rho_x-\rho)(\vx^{k+1}-\vx)^\top\vA^\top\vr^{k+1}
-\rho_x(\vx^{k+1}-\vx)^\top\vA^\top\vB(\vy^{k+1}-\vy^k).\nonumber
\end{eqnarray}
Hence,
we can rewrite \eqref{ineq-k0-x} equivalently into \eqref{ineq-k1-x}.
Through the same arguments, one can show \eqref{ineq-k1-y}, thus completing the proof.

\subsection{Proof of Lemma \ref{lem:xy-rate}}
Letting $x=x^*,y=y^*$ in \eqref{eq:xy-phi}, we have for any $\lambda$ that 
\begin{eqnarray}
& & \EE[h(x^*,y^*,\vlam)]\nonumber\\
&\ge & \EE\left[\Phi(\hat{\vx},\hat{\vy})-\Phi(\vx^*,\vy^*)+(\hat{x}-x^*)^{\top}(-A^{\top}\hat{\lambda})
+(\hat{y}-y^*)^{\top}(-B^{\top}\hat{\lambda})+(\hat{\lambda}-\lambda)^{\top}(A\hat{x}+B\hat{y}-b)\right]\nonumber \\
&=& \EE\left[\Phi(\hat{\vx},\hat{\vy})-\Phi(\vx^*,\vy^*)+\langle\hat{\lambda},Ax^*+By^*-b\rangle-\langle\lambda,A\hat{x}+B\hat{y}-b\rangle\right]\nonumber\\
&=&\EE\left[\Phi(\hat{\vx},\hat{\vy})-\Phi(\vx^*,\vy^*)-\langle\lambda,A\hat{x}+B\hat{y}-b\rangle\right],
\label{eq:xy-phi-star}
\end{eqnarray}
where the last equality follows from the feasibility of $(x^*,y^*)$.
For any $\gamma>0$, restricting $\lambda$ in $\mathcal{B}_\gamma$, we have
$$\EE[h(x^*,y^*,\vlam)]\le \sup\limits_{\lambda\in\cB_\gamma}h(x^*,y^*,\vlam).$$
Hence, letting $\lambda=-\frac{\gamma(A\hat{x}+B\hat{y}-b)}{\|A\hat{x}+B\hat{y}-b\|}\in\mathcal{B}_\gamma$ in \eqref{eq:xy-phi-star} gives the desired result.

\subsection{Proof of Inequalities \eqref{bd-x-s1} and \eqref{bd-x-s2} with $\alpha_k=\frac{\alpha_0}{\sqrt{k}}$}\label{app:bd-x-s}
We have
$\beta_k=\frac{\alpha_0}{\rho\big(\sqrt{k}-(1-\theta)\sqrt{k-1}\big)},$
and
\begin{eqnarray*}
& &\frac{\alpha_{k-1}}{\alpha_{k}}\beta_{k}+(1-\theta)\beta_{k+1}-\frac{\alpha_k}{\alpha_{k+1}}\beta_{k+1}-(1-\theta)\beta_{k}\\
&=&\frac{\alpha_0}{\rho}\left[\left(\frac{\sqrt{k}}{\sqrt{k-1}}-(1-\theta)\right)\frac{1}{\big(\sqrt{k}-(1-\theta)\sqrt{k-1}\big)}
-\left(\frac{\sqrt{k+1}}{\sqrt{k}}-(1-\theta)\right)\frac{1}{\big(\sqrt{k+1}-(1-\theta)\sqrt{k}\big)}\right]\cr
&=:&\frac{\alpha_0}{\rho}[\psi(k)-\psi(k+1)].
\end{eqnarray*}
By elementary calculus, we have
\begin{eqnarray*}
\psi'(k)&=&\frac{\frac{\sqrt{k-1}}{\sqrt{k}}-\frac{\sqrt{k}}{\sqrt{k-1}}}{2(k-1)}\frac{1}{\big(\sqrt{k}-(1-\theta)\sqrt{k-1}\big)}\\
& & +\left(\frac{\sqrt{k}}{\sqrt{k-1}}-(1-\theta)\right)\frac{-1}{2\big(\sqrt{k}-(1-\theta)\sqrt{k-1}\big)^2}\left(\frac{1}{\sqrt{k}}-\frac{1-\theta}{\sqrt{k-1}}\right)\\
&=&\frac{1}{2(k-1)(\sqrt{k}-(1-\theta)\sqrt{k-1})}\left[\frac{\sqrt{k-1}}{\sqrt{k}}-\frac{\sqrt{k}}{\sqrt{k-1}}-\sqrt{k-1}\left(\frac{1}{\sqrt{k}}-\frac{1-\theta}{\sqrt{k-1}}\right)\right]\\
&=&\frac{1}{2(k-1)(\sqrt{k}-(1-\theta)\sqrt{k-1})}\left((1-\theta)-\frac{\sqrt{k}}{\sqrt{k-1}}\right)<0.
\end{eqnarray*}
Hence, $\psi(k)$ is decreasing with respect to $k$, and thus \eqref{bd-x-s1} holds.

When $\alpha_k=\frac{\alpha_0}{\sqrt{k}}$, \eqref{bd-x-s2} becomes
$$\frac{\alpha_0}{2\rho\sqrt{t}}\ge\left|\left(\frac{\sqrt{t}}{\sqrt{t-1}}-(1-\theta)\right)\frac{\alpha_0}{\rho(\sqrt{t}-(1-\theta)\sqrt{t-1})}-\frac{\alpha_0}{\rho\sqrt{t}}\right|,$$
which is equivalent to
$$\frac{1}{2}\ge \frac{\sqrt{t}}{\sqrt{t-1}}-1 \Longleftrightarrow t\ge \frac{9}{5}.$$
This completes the proof.

\section{Proofs of Theorems} \label{sec:app-B}
In this section, we give the technical details for showing all theorems. 

\subsection{Proof of Theorem \ref{thm:rate-3X}}
Taking expectation over both sides of \eqref{ineq-k1-x} and summing it over $k=0$ through $t$, we have
\begin{eqnarray}\label{ineq-k2-x}
& &\EE\big[F(\vx^{t+1})-F(\vx)+(\vx^{t+1}-\vx)^\top(-\vA^\top\vlam^{t+1})\big]+(1-\theta)\rho_x\EE(\vx^{t+1}-\vx)^\top\vA^\top\vr^{t+1}\cr
& &+\theta\sum_{k=0}^{t-1}\EE\big[F(\vx^{k+1})-F(\vx)+(\vx^{k+1}-\vx)^\top(-\vA^\top\vlam^{k+1})\big]
-\sum_{k=0}^t\rho_x\EE(\vx^{k+1}-\vx)^\top\vA^\top\vB(\vy^{k+1}-\vy^k)\cr
& &+\sum_{k=0}^t\EE(\vx^{k+1}-\vx)^\top\tilde{\vP}(\vx^{k+1}-\vx^k)-\frac{L_f}{2}\sum_{k=0}^t\EE\|\vx^k-\vx^{k+1}\|^2\cr
&\le & (1-\theta)\left[F(\vx^0)-F(\vx)+(\vx^{0}-\vx)^\top(-\vA^\top\vlam^0)+
\rho_x(\vx^{0}-\vx)^\top\vA^\top\vr^{0}\right],
\end{eqnarray}
where we have used $\frac{n}{N}=\theta$, the condition in \eqref{para-rho} and the definition of $\tilde{P}$ in  \eqref{matPQtilde}.
Similarly, taking expectation over both sides of \eqref{ineq-k1-y}, summing it over $k=0$ through $t$, we have
\begin{eqnarray}\label{ineq-k2-y}
& &\EE\big[G(\vy^{t+1})-G(\vy)+(\vy^{t+1}-\vy)^\top(-\vB^\top\vlam^{t+1})\big]+
(1-\theta)\rho_y\EE(\vy^{t+1}-\vy)^\top\vB^\top\vr^{t+1}\cr
& &+\theta\sum_{k=0}^{t-1}\EE\big[G(\vy^{k+1})-G(\vy)+(\vy^{k+1}-\vy)^\top(-\vB^\top\vlam^{k+1})\big]\cr
& &+\sum_{k=0}^t\EE(\vy^{k+1}-\vy)^\top\tilde{\vQ}(\vy^{k+1}-\vy^k)-\frac{L_g}{2}\sum_{k=0}^t\EE\|\vy^k-\vy^{k+1}\|^2\cr
&\le & (1-\theta)\left[G(\vy^0)-G(\vy)+(\vy^{0}-\vy)^\top(-\vB^\top\vlam^0)+
\rho_y(\vy^{0}-\vy)^\top\vB^\top\vr^{0}\right]\\
& &+(1-\theta)\sum_{k=0}^t\EE\rho_y(\vy^{k}-\vy)^\top\vB^\top\vA(\vx^{k+1}-\vx^k).\nonumber
\end{eqnarray}

Recall $\vlam^{k+1}=\vlam^k-\rho\vr^{k+1}$, thus
\begin{equation}\label{eq-lam}
(\vlam^{k+1}-\vlam)^\top \vr^{k+1} =-\frac{1}{\rho}(\vlam^{k+1}-\vlam)^\top(\vlam^{k+1}-\vlam^k),
\end{equation}
where $\lambda$ is an arbitrary vector and possibly random.
Denote $\tilde{\vlam}^{t+1}=\vlam^t-\rho_x\vr^{t+1}.$
Then similar to \eqref{eq-lam}, we have
\begin{equation}\label{eq-lamt}
(\tilde{\vlam}^{t+1}-\vlam)^\top \vr^{t+1} =-\frac{1}{\rho_x}(\tilde{\vlam}^{t+1}-\vlam)^\top(\tilde{\vlam}^{t+1}-\vlam^t).
\end{equation}

Summing \eqref{ineq-k2-x} and \eqref{ineq-k2-y} together and using \eqref{eq-lam} and \eqref{eq-lamt}, we have:
\begin{eqnarray}\label{ineq-k3-w}
& &\EE\left[\Phi(\vx^{t+1},\vy^{t+1})-\Phi(\vx,\vy)+(\tilde{\vw}^{t+1}-\vw)^\top H(\tilde{\vw}^{t+1})+\frac{1}{\rho_x}(\tilde{\vlam}^{t+1}-\vlam)^\top(\tilde{\vlam}^{t+1}-\vlam^t)\right]\cr
& &+\theta\sum_{k=0}^{t-1}\EE\left[\Phi(\vx^{k+1},\vy^{k+1})-\Phi(\vx,\vy)+(\vw^{k+1}-\vw)^\top H(\vw^{k+1})+\frac{1}{\rho}(\vlam^{k+1}-\vlam)^\top(\vlam^{k+1}-\vlam^k)\right]\cr
&\le & (1-\theta)\left[F(\vx^0)-F(\vx)+(\vx^{0}-\vx)^\top(-\vA^\top\vlam^0)+
\rho_x(\vx^{0}-\vx)^\top\vA^\top\vr^{0}\right]\cr
& &+(1-\theta)\left[G(\vy^0)-G(\vy)+(\vy^{0}-\vy)^\top(-\vB^\top\vlam^0)+
\rho_y(\vy^{0}-\vy)^\top\vB^\top\vr^{0}\right]\cr
& &+\sum_{k=0}^t\rho_x\EE(\vx^{k+1}-\vx)^\top\vA^\top\vB(\vy^{k+1}-\vy^k)
+(1-\theta)\sum_{k=0}^t\rho_y\EE(\vy^{k}-\vy)^\top\vB^\top\vA(\vx^{k+1}-\vx^k)\cr
& & -\sum_{k=0}^t\EE(\vx^{k+1}-\vx)^\top\tilde{\vP}(\vx^{k+1}-\vx^k)+\frac{L_f}{2}\sum_{k=0}^t\EE\|\vx^k-\vx^{k+1}\|^2\cr
& & -\sum_{k=0}^t\EE(\vy^{k+1}-\vy)^\top\tilde{\vQ}(\vy^{k+1}-\vy^k)+\frac{L_g}{2}\sum_{k=0}^t\EE\|\vy^k-\vy^{k+1}\|^2,
\end{eqnarray}
where we have used $\Phi(\vx,\vy)=F(\vx)+G(\vy)$ and the definition of $H$ given in \eqref{w-H}.

When $B=0$ and $y^k\equiv y^0$, \eqref{ineq-k3-w} reduces to
\begin{eqnarray*}
& &\EE\left[F(\vx^{t+1})-F(\vx)+(\tilde{\vw}^{t+1}-\vw)^\top H(\tilde{\vw}^{t+1})+\frac{1}{\rho_x}(\tilde{\vlam}^{t+1}-\vlam)^\top(\tilde{\vlam}^{t+1}-\vlam^t)\right]\cr
& &+\theta\sum_{k=0}^{t-1}\EE\left[F(\vx^{k+1})-F(\vx)+(\vw^{k+1}-\vw)^\top H(\vw^{k+1})+\frac{1}{\rho}(\vlam^{k+1}-\vlam)^\top(\vlam^{k+1}-\vlam^k)\right]\cr
&\le & (1-\theta)\left[F(\vx^0)-F(\vx)+(\vx^{0}-\vx)^\top(-\vA^\top\vlam^0)+
\rho_x(\vx^{0}-\vx)^\top\vA^\top\vr^{0}\right]\cr
& &-\sum_{k=0}^t\EE(\vx^{k+1}-\vx)^\top\tilde{\vP}(\vx^{k+1}-\vx^k)+\frac{L_f}{2}\sum_{k=0}^t\EE\|\vx^k-\vx^{k+1}\|^2.
\end{eqnarray*}
Using \eqref{uv-cross} and noting $\theta=\frac{\rho}{\rho_x}$, from the above inequality after cancelling terms we have
\begin{eqnarray}\label{ineq-k3-w-0y}
& &\EE\left[F(\vx^{t+1})-F(\vx)+(\tilde{\vw}^{t+1}-\vw)^\top H(\tilde{\vw}^{t+1})\right]+\theta\sum_{k=0}^{t-1}\EE\left[F(\vx^{k+1})-F(\vx)+(\vw^{k+1}-\vw)^\top H(\vw^{k+1})\right]\nonumber\\
& &+\frac{1}{2\rho_x}\EE\left[\|\tilde{\vlam}^{t+1}-\vlam\|^2-\|\vlam^0-\vlam\|^2+\|\tilde{\vlam}^{t+1}-\vlam^t\|^2+\sum_{k=0}^{t-1}\|\vlam^{k+1}-\vlam^k\|^2\right]\cr
&\le & (1-\theta)\left[F(\vx^0)-F(\vx)+(\vx^{0}-\vx)^\top(-\vA^\top\vlam^0)+
\rho_x(\vx^{0}-\vx)^\top\vA^\top\vr^{0}\right]\nonumber\\
& &-\frac{1}{2}\EE\left[\|x^{t+1}-x\|_{\tilde{P}}^2-\|x^0-x\|_{\tilde{P}}^2+\sum_{k=0}^t\|x^{k+1}-x^k\|_{\tilde{P}}^2\right]+\frac{L_f}{2}\sum_{k=0}^t\EE\|\vx^k-\vx^{k+1}\|^2.
\end{eqnarray}
For any feasible $x$, we note $\tilde{\vlam}^{t+1}-\vlam^t=\rho_xA(x^{t+1}-x)$ and thus
\begin{equation}
\frac{1}{\rho_x}\|\tilde{\vlam}^{t+1}-\vlam^t\|^2=\rho_x\|x^{t+1}-x\|_{A^\top A}^2\label{temp-lam}.
\end{equation}
In addition, since $x^{k+1}$ and $x^k$ differ only on the index set $I_k$, we have by recalling $\tilde{P}=\hat{P}-\rho_x A^\top A$ that
\begin{equation}\label{temp-x}
\|x^{k+1}-x^k\|_{\tilde{P}}^2-L_f\|x^{k+1}-x^k\|^2=\|x_{I_k}^{k+1}-x_{I_k}^k\|_{\hat{P}_{I_k}}^2-\|x_{I_k}^{k+1}-x_{I_k}^k\|_{\rho_x A_{I_k}^\top A_{I_k}}^2-L_f\|x_{I_k}^{k+1}-x_{I_k}^k\|^2.
\end{equation}
Plugging \eqref{temp-lam} and \eqref{temp-x} into \eqref{ineq-k3-w-0y}, and using \eqref{para-mat-3X} leads to
\begin{eqnarray*}
& &\EE\left[F(\vx^{t+1})-F(\vx)+(\tilde{\vw}^{t+1}-\vw)^\top H(\tilde{\vw}^{t+1})\right]+\theta\sum_{k=0}^{t-1}\EE\left[F(\vx^{k+1})-F(\vx)+(\vw^{k+1}-\vw)^\top H(\vw^{k+1})\right]\cr
&\le & (1-\theta)\left[F(\vx^0)-F(\vx)+(\vx^{0}-\vx)^\top(-\vA^\top\vlam^0)+
\rho_x(\vx^{0}-\vx)^\top\vA^\top\vr^{0}\right]
+\frac{1}{2\rho_x}\EE\|\vlam^0-\vlam\|^2+\frac{1}{2}\|x^0-x\|_{\tilde{P}}^2.
\end{eqnarray*}
The desired result follows from $\lambda^0=0$, and Lemmas \ref{lem:xy-rate} and \ref{equiv-rate} with $\gamma=\max\{1+\|\lambda^*\|, 2\|\lambda^*\|\}$.

\subsection{Proof of Theorem \ref{thm:rate-1Yw}}

It follows from \eqref{ineq-k1-y} with $\rho_y=\rho$ and $m=M$ that (recall the definition of $\tilde{Q}$ in  \eqref{matPQtilde}) for any $\vy\in\cY$,
\begin{equation}\label{ineq2-k2}
G(\vy^{k+1})-G(\vy)-\frac{L_g}{2}\|\vy^k-\vy^{k+1}\|^2+(\vy^{k+1}-\vy)^\top(-\vB^\top\vlam^{k+1})+(\vy^{k+1}-\vy)^\top
\tilde{\vQ}(\vy^{k+1}-\vy^k)\le 0.
\end{equation}
Similar to \eqref{ineq2-k2}, and recall the definition of $\tilde{y}^{t+1}$, we have for any $\vy\in\cY$,
\begin{equation}\label{ineq2-k2-tildey}
G(\tilde{\vy}^{t+1})-G(\vy)-\frac{L_g}{2}\|\tilde{\vy}^{t+1}-\vy^t\|^2+(\tilde{\vy}^{t+1}-\vy)^\top(-\vB^\top\tilde{\vlam}^{t+1})+\theta(\tilde{\vy}^{t+1}-\vy)^\top
\tilde{\vQ}(\vy^{t+1}-\vy^t)\le 0,
\end{equation}
where
\begin{equation}\label{tilde-lam2}
\tilde{\vlam}^{t+1}=\vlam^t-\rho_x(\vA\vx^{t+1}+\vB\tilde{\vy}^{t+1}-\vb).
\end{equation}
Adding \eqref{ineq2-k2} and \eqref{ineq2-k2-tildey} to \eqref{ineq-k2-x} and using the formula of $\vlam^k$ gives
\begin{eqnarray}\label{ineq-k2-x-1Yw}
& &\EE\left[F(\vx^{t+1})-F(\vx)+(\vx^{t+1}-\vx)^\top(-\vA^\top\tilde{\vlam}^{t+1})\right]
+\EE\left(\tilde{\vlam}^{t+1}-\vlam\right)^\top\left(\vA\vx^{t+1}+\vB\tilde{\vy}^{t+1}-\vb+\frac{1}{\rho_x}(\tilde{\vlam}^{t+1}-\vlam^t)\right)\nonumber\\
& &+\EE\left[G(\tilde{\vy}^{t+1})-G(\vy)+(\tilde{\vy}^{t+1}-\vy)^\top(-\vB^\top\tilde{\vlam}^{t+1})+\theta(\tilde{\vy}^{t+1}-\vy)^\top
\tilde{\vQ}(\tilde{\vy}^{t+1}-\vy^k)\right]-\frac{L_g}{2}\EE\|\tilde{\vy}^{t+1}-\vy^t\|^2\nonumber\\
& &+\theta\sum_{k=0}^{t-1}\EE\left[F(\vx^{k+1})-F(\vx)+(\vx^{k+1}-\vx)^\top(-\vA^\top\vlam^{k+1})\right]\nonumber\\
& &-\sum_{k=0}^{t-1}\rho_x\EE(\vx^{k+1}-\vx)^\top\vA^\top\vB(\vy^{k+1}-\vy^k)
-\rho_x\EE(\vx^{t+1}-\vx)^\top\vA^\top\vB(\tilde{\vy}^{t+1}-\vy^t)\nonumber\\
& &+\theta\sum_{k=0}^{t-1}\EE\left[G(\vy^{k+1})-G(\vy)-\frac{L_g}{2}\|\vy^k-\vy^{k+1}\|^2+(\vy^{k+1}-\vy)^\top(-\vB^\top\vlam^{k+1})+(\vy^{k+1}-\vy)^\top
\tilde{\vQ}(\vy^{k+1}-\vy^k)\right]\nonumber\\
& &+\theta\sum_{k=0}^{t-1}\EE(\vlam^{k+1}-\vlam)^\top\left(\vr^{k+1}+\frac{1}{\rho}(\vlam^{k+1}-\vlam^k)\right)\nonumber\\
& \le & (1-\theta)\left[F(\vx^0)-F(\vx)+(\vx^{0}-\vx)^\top(-\vA^\top\vlam^0)+
\rho_x(\vx^{0}-\vx)^\top\vA^\top\vr^{0}\right]\nonumber\\
& &-\sum_{k=0}^t\EE(\vx^{k+1}-\vx)^\top\tilde{\vP}(\vx^{k+1}-\vx^k)+\frac{L_f}{2}\sum_{k=0}^t\EE\|\vx^k-\vx^{k+1}\|^2.
\end{eqnarray}
By the notation in \eqref{w-H} and using \eqref{feas-x}, \eqref{ineq-k2-x-1Yw} can be written into
\begin{eqnarray*}
& &\EE\left[\Phi(\vx^{t+1},\tilde{\vy}^{t+1})-\Phi(\vx,\vy)+(\tilde{\vw}^{t+1}-\vw)^\top H(\tilde{\vw}^{t+1})\right]\\
& & +\theta\sum_{k=0}^{t-1}\EE\left[\Phi(\vx^{k+1},\vy^{k+1})-\Phi(\vx,\vy)+(\vw^{k+1}-\vw)^\top H(\vw^{k+1})\right]
\nonumber\\
& &+\theta\sum_{k=0}^{t-1}\EE(\vy^{k+1}-\vy)^\top
\tilde{\vQ}(\vy^{k+1}-\vy^k)+\theta\EE(\tilde{\vy}^{t+1}-\vy)^\top
\tilde{\vQ}(\tilde{\vy}^{t+1}-\vy^t)\nonumber\\
& &-\sum_{k=0}^{t-1}\rho_x\EE\left(\frac{1}{\rho}(\vlam^k-\vlam^{k+1})^\top\vB(\vy^{k+1}-\vy^k)
-(\vy^{k+1}-\vy)^\top\vB^\top\vB(\vy^{k+1}-\vy^k)\right)\nonumber\\
& &-\rho_x\EE\left(\frac{1}{\rho_x}(\vlam^t-\tilde{\vlam}^{t+1})^\top\vB(\tilde{\vy}^{t+1}-\vy^t)
-(\tilde{\vy}^{t+1}-\vy)^\top\vB^\top\vB(\tilde{\vy}^{t+1}-\vy^t)\right)\nonumber\\
& &+\frac{\theta}{\rho}\EE(\tilde{\vlam}^{t+1}-\vlam)^\top\left(\tilde{\vlam}^{t+1}-\vlam^t\right)
+\frac{\theta}{\rho}\sum_{k=0}^{t-1}\EE(\vlam^{k+1}-\vlam)^\top\left(\vlam^{k+1}-\vlam^k\right)\cr
&\le & (1-\theta)\left[F(\vx^0)-F(\vx)+(\vx^{0}-\vx)^\top(-\vA^\top\vlam^0)+
\rho_x(\vx^{0}-\vx)^\top\vA^\top\vr^{0}\right]\nonumber\\
& &-\sum_{k=0}^t\EE(\vx^{k+1}-\vx)^\top\tilde{\vP}(\vx^{k+1}-\vx^k)+\frac{L_f}{2}\sum_{k=0}^t\EE\|\vx^k-\vx^{k+1}\|^2\nonumber\\
& &+\frac{\theta L_g}{2}\sum_{k=0}^{t-1}\EE\|\vy^k-\vy^{k+1}\|^2+\frac{L_g}{2}\EE\|\tilde{\vy}^{t+1}-\vy^t\|^2.
\end{eqnarray*}
Now use \eqref{uv-cross} to derive from the above inequality that
\begin{eqnarray}\label{ineq-k3-x-1Yw}
& &\EE\left[\Phi(\vx^{t+1},\tilde{\vy}^{t+1})-\Phi(\vx,\vy)+(\tilde{\vw}^{t+1}-\vw)^\top H(\tilde{\vw}^{t+1})\right] \nonumber \\
& & +\theta\sum_{k=0}^{t-1}\EE\left[\Phi(\vx^{k+1},\vy^{k+1})-\Phi(\vx,\vy)+(\vw^{k+1}-\vw)^\top H(\vw^{k+1})\right]
\nonumber\\
& &+\frac{\theta}{2}\left(\EE\|\tilde{\vy}^{t+1}-\vy\|_{\tilde{\vQ}}^2-\|\vy^0-\vy\|_{\tilde{\vQ}}^2\right)
+\frac{\theta}{2}\sum_{k=0}^{t-1}\EE\|\vy^{k+1}-\vy^k\|_{\tilde{\vQ}}^2+\frac{\theta}{2}\EE\|\tilde{\vy}^{t+1}-\vy^t\|_{\tilde{\vQ}}^2\nonumber\\
& &+\frac{\rho_x}{2}\left(\EE\|\tilde{\vy}^{t+1}-\vy\|_{\vB^\top\vB}^2-\|\vy^0-\vy\|_{\vB^\top\vB}^2\right)
+\frac{\rho_x}{2}\sum_{k=0}^{t-1}\EE\|\vy^{k+1}-\vy^k\|_{\vB^\top\vB}^2+\frac{\rho_x}{2}\EE\|\tilde{\vy}^{t+1}-\vy^t\|_{\vB^\top\vB}^2\nonumber\\
& &-\sum_{k=0}^{t-1}\EE\frac{\rho_x}{\rho}\left(\vlam^k-\vlam^{k+1}\right)^\top\vB(\vy^{k+1}-\vy^k)
-\EE(\vlam^t-\tilde{\vlam}^{t+1})^\top\vB(\tilde{\vy}^{t+1}-\vy^t)\nonumber\\
& &+\frac{\theta}{2\rho}\left(\EE\|\tilde{\vlam}^{t+1}-\vlam\|^2-\|\vlam^0-\vlam\|^2\right)
+\frac{\theta}{2\rho}\sum_{k=0}^{t-1}\EE\|\vlam^{k+1}-\vlam^k\|^2+\frac{\theta}{2\rho}\EE\|\tilde{\vlam}^{t+1}-\vlam^t\|^2\nonumber\\
&\le & (1-\theta)\left[F(\vx^0)-F(\vx)+(\vx^{0}-\vx)^\top(-\vA^\top\vlam^0)+
\rho_x(\vx^{0}-\vx)^\top\vA^\top\vr^{0}\right]\nonumber\\
& &-\frac{1}{2}\left[\EE\|\vx^{t+1}-\vx\|^2_{\tilde{\vP}}-\|\vx^0-\vx\|^2_{\tilde{\vP}}
+\sum_{k=0}^t\EE\|\vx^k-\vx^{k+1}\|^2_{\tilde{\vP}}\right]+\frac{L_f}{2}\sum_{k=0}^t\EE\|\vx^k-\vx^{k+1}\|^2\nonumber\\
& &+\frac{\theta L_g}{2}\sum_{k=0}^{t-1}\EE\|\vy^k-\vy^{k+1}\|^2+\frac{L_g}{2}\EE\|\tilde{\vy}^{t+1}-\vy^t\|^2.
\end{eqnarray}
Note that for $k\le t-1$,
$$-\frac{\rho_x}{\rho}(\vlam^k-\vlam^{k+1})^\top\vB(\vy^{k+1}-\vy^k)+\frac{\theta}{2\rho}\|\vlam^{k+1}-\vlam^k\|^2\ge-\frac{\rho}{2\theta^3}\|\vy^{k+1}-\vy^k\|_{\vB^\top\vB}^2$$
and
$$-(\vlam^t-\tilde{\vlam}^{t+1})^\top\vB(\tilde{\vy}^{t+1}-\vy^t)+\frac{\theta}{2\rho}\|\tilde{\vlam}^{t+1}-\vlam^t\|^2\ge-\frac{\rho}{2\theta}\|\tilde{\vy}^{t+1}-\vy^t\|_{\vB^\top\vB}^2.$$

Because $\tilde{\vP}, {\tilde{\vQ}}$ and $\rho$ satisfy \eqref{para-mat-1Yw},
we have from \eqref{ineq-k3-x-1Yw} that
\begin{eqnarray}\label{ineq-k4-x-1Yw}
& &\EE\left[\Phi(\vx^{t+1},\tilde{\vy}^{t+1})-\Phi(\vx,\vy)+(\tilde{\vw}^{t+1}-\vw)^\top H(\tilde{\vw}^{t+1})\right] \nonumber \\
& & +\theta\sum_{k=0}^{t-1}\EE\left[\Phi(\vx^{k+1},\vy^{k+1})-\Phi(\vx,\vy)+(\vw^{k+1}-\vw)^\top H(\vw^{k+1})\right]
\nonumber\cr
&\le & (1-\theta)\left[F(\vx^0)-F(\vx)+(\vx^{0}-\vx)^\top(-\vA^\top\vlam^0)+
\rho_x(\vx^{0}-\vx)^\top\vA^\top\vr^{0}\right]\cr
& &+\frac{1}{2}\|\vx^0-\vx\|^2_{\tilde{\vP}}+\frac{\theta}{2}\|\vy^0-\vy\|_{\tilde{\vQ}}^2+\frac{\rho}{2\theta}\|\vy^0-\vy\|_{\vB^\top\vB}^2+\frac{\theta}{2\rho}\EE\|\vlam^0-\vlam\|^2.
\end{eqnarray}
Similar to Theorem \ref{thm:rate-cvx}, from the convexity of $\Phi$ and \eqref{prop-mas-H}, we have
\begin{eqnarray}
& &(1+\theta t)\EE\big[\Phi(\hat{\vx}^t,\hat{\vy}^t)-\Phi(\vx,\vy)+(\hat{\vw}^{t+1}-\vw)^\top H({\vw})\big]\cr
&\le & (1-\theta)\big[F(\vx^0)-F(\vx)+(\vx^{0}-\vx)^\top(-\vA^\top\vlam^0)+
\rho_x(\vx^{0}-\vx)^\top\vA^\top\vr^{0}\big]\cr
& &+\frac{1}{2}\|\vx^0-\vx\|^2_{\tilde{\vP}}+\frac{\theta}{2}\|\vy^0-\vy\|_{\tilde{\vQ}}^2+\frac{\rho}{2\theta}\|\vy^0-\vy\|_{\vB^\top\vB}^2+\frac{\theta}{2\rho}\EE\|\vlam^0-\vlam\|^2.
\end{eqnarray}
Noting $\lambda^0=0$ and $(\vx^{0}-\vx)^\top\vA^\top\vr^{0}\le \frac{1}{2}\big[\|x^0-x\|_{A^\top A}+\|r^0\|^2\big]$, and using Lemmas \ref{lem:xy-rate} and \ref{equiv-rate} with $\gamma=\max\{1+\|\lambda^*\|, 2\|\lambda^*\|\}$, we obtain the result \eqref{eq:rate-iym}.

\subsection{Proof of Theorem \ref{thm:rate-cvx}}
Using \eqref{feas-x} and \eqref{feas-y}, applying \eqref{uv-cross} to the cross terms, and also noting the definition of $\tilde{P}$ and $\tilde{Q}$ in \eqref{matPQtilde}, we have
\begin{eqnarray}\label{ineq-k4-w-sub1}
& &-\frac{\theta}{\rho}\EE\left[\sum_{k=0}^{t-1}(\vlam^{k+1}-\vlam)^\top(\vlam^{k+1}-\vlam^k)
+(\tilde{\vlam}^{t+1}-\vlam)^\top(\tilde{\vlam}^{t+1}-\vlam^t)\right]\cr
& &+\sum_{k=0}^t\rho_x\EE(\vx^{k+1}-\vx)^\top\vA^\top\vB(\vy^{k+1}-\vy^k)
+(1-\theta)\sum_{k=0}^t\rho_y\EE(\vy^{k}-\vy)^\top\vB^\top\vA(\vx^{k+1}-\vx^k)\cr
& & -\sum_{k=0}^t\EE(\vx^{k+1}-\vx)^\top\tilde{\vP}(\vx^{k+1}-\vx^k)+\frac{L_f}{2}\sum_{k=0}^t\EE\|\vx^k-\vx^{k+1}\|^2\cr
& & -\sum_{k=0}^t\EE(\vy^{k+1}-\vy)^\top\tilde{\vQ}(\vy^{k+1}-\vy^k)+\frac{L_g}{2}\sum_{k=0}^t\EE\|\vy^k-\vy^{k+1}\|^2\cr
&=&-\frac{\theta}{2\rho}\EE\left[\|\tilde{\vlam}^{t+1}-\vlam\|^2-\|\vlam^0-\vlam\|^2+\sum_{k=0}^{t-1}\|\vlam^{k+1}-\vlam^k\|^2+\|\tilde{\vlam}^{t+1}-\vlam^t\|^2\right]\cr
& &+\frac{\rho_x}{\rho}\sum_{k=0}^t \EE(\vlam^k-\vlam^{k+1})^\top\vB(\vy^{k+1}-\vy^k)+\frac{(1-\theta)\rho_y}{\rho}\sum_{k=0}^t\EE(\vlam^{k-1}-\vlam^{k})^\top\vA(\vx^{k+1}-\vx^k)\cr
& &-\frac{\theta\rho_y}{2}\EE\left(\|\vx^0-\vx\|_{\vA^\top\vA}^2-\|\vx^{t+1}-\vx\|_{\vA^\top\vA}^2\right)+\frac{(2-\theta)\rho_y}{2}\sum_{k=0}^t\EE\|\vx^{k+1}-\vx^k\|_{\vA^\top\vA}^2\cr
& &-\frac{1}{2}\EE\left(\|\vx^{t+1}-\vx\|_{\hat{\vP}}^2-\|\vx^0-\vx\|_{\hat{\vP}}^2
+\sum_{k=0}^t\|\vx^{k+1}-\vx^k\|_{\hat{\vP}}^2\right)+\frac{L_f}{2}\sum_{k=0}^t\EE\|\vx^k-\vx^{k+1}\|^2\cr
& &-\frac{1}{2}\EE\left(\|\vy^{t+1}-\vy\|_{\hat{\vQ}}^2-\|\vy^0-\vy\|_{\hat{\vQ}}^2
+\sum_{k=0}^t\|\vy^{k+1}-\vy^k\|_{\hat{\vQ}}^2\right)+\frac{L_g}{2}\sum_{k=0}^t\EE\|\vy^k-\vy^{k+1}\|^2,
\end{eqnarray}
where we have used the conditions in \eqref{para-rho}.
By Young's inequality, we have that for $0\le k\le t$,
\begin{eqnarray}\label{bd-lam-y}
& &\frac{\rho_x}{\rho}(\vlam^k-\vlam^{k+1})^\top\vB(\vy^{k+1}-\vy^k)-\frac{\theta}{2\rho}\frac{1}{2-\theta}\|\vlam^{k+1}-\vlam^k\|^2\cr
&\le & \frac{\rho}{\theta}\frac{2-\theta}{2}\frac{\rho_x^2}{\rho^2}\|\vB(\vy^{k+1}-\vy^k)\|^2
\overset{\eqref{para-rho}}= \frac{(2-\theta)\rho_y}{2\theta^2}\|\vy^{k+1}-\vy^k\|_{\vB^\top\vB}^2,
\end{eqnarray}
and for $1\le k\le t$,
\begin{eqnarray}\label{bd-lam-x}
& &\frac{(1-\theta)\rho_y}{\rho}(\vlam^{k-1}-\vlam^{k})^\top\vA(\vx^{k+1}-\vx^k)-\frac{\theta}{2\rho}\frac{1-\theta}{2-\theta}\|\vlam^{k-1}-\vlam^k\|^2\cr
&\le & (1-\theta)\frac{\rho}{\theta}\frac{(2-\theta)\rho_y^2}{2\rho^2}\|\vA(\vx^{k+1}-\vx^k)\|^2\overset{\eqref{para-rho}}=\frac{(1-\theta)(2-\theta)}{2\theta^2}\rho_x\|\vx^{k+1}-\vx^k\|_{\vA^\top\vA}^2.
\end{eqnarray}
Plugging \eqref{bd-lam-y} and \eqref{bd-lam-x} and also noting $\|\tilde{\vlam}^{t+1}-\vlam^t\|^2\ge \|\vlam^{t+1}-\vlam^t\|^2$, we can upper bound the right hand side of \eqref{ineq-k4-w-sub1} by
\begin{eqnarray}
& &-\frac{\theta}{2\rho}\EE\left[\|\tilde{\vlam}^{t+1}-\vlam\|^2-\|\vlam^0-\vlam\|^2\right]-\frac{\theta\rho_y}{2}\EE\left(\|\vx^0-\vx\|_{\vA^\top\vA}^2-\|\vx^{t+1}-\vx\|_{\vA^\top\vA}^2\right)\cr
& &+\left(\frac{(1-\theta)(2-\theta)}{2\theta^2}\rho_x+\frac{(2-\theta)\rho_y}{2}\right)\sum_{k=0}^t \EE\|\vx^{k+1}-\vx^k\|_{\vA^\top\vA}^2+\frac{(2-\theta)\rho_y}{2\theta^2}\sum_{k=0}^t\EE\|\vy^{k+1}-\vy^k\|_{\vB^\top\vB}^2\cr
& &-\frac{1}{2}\EE\left(\|\vx^{t+1}-\vx\|_{\hat{\vP}}^2-\|\vx^0-\vx\|_{\hat{\vP}}^2
+\sum_{k=0}^t\|\vx^{k+1}-\vx^k\|_{\hat{\vP}}^2\right)+\frac{L_f}{2}\sum_{k=0}^t\EE\|\vx^k-\vx^{k+1}\|^2\cr
& &-\frac{1}{2}\EE\left(\|\vy^{t+1}-\vy\|_{\hat{\vQ}}^2-\|\vy^0-\vy\|_{\hat{\vQ}}^2
+\sum_{k=0}^t\|\vy^{k+1}-\vy^k\|_{\hat{\vQ}}^2\right)+\frac{L_g}{2}\sum_{k=0}^t\EE\|\vy^k-\vy^{k+1}\|^2\cr
&\overset{\eqref{para-mat}}\le &\frac{1}{2}\left(\|\vx^0-\vx\|_{\hat{\vP}-\theta\rho_xA^\top A}^2+\|\vy^0-\vy\|_{\hat{\vQ}}^2\right)+\frac{\theta}{2\rho}\EE\|\lambda^0-\lambda\|^2.\label{ineq-k4-w-sub2}
\end{eqnarray}
In addition, note that
\begin{align*}
&\theta\|x^{t+1}-x\|_{A^\top A}^2=\frac{n}{N}\left\|\sum_{i=1}^N A_i(x_i^{t+1}-x_i)\right\|^2\le n\sum_{i=1}^N\|x_i^{t+1}-x_i\|_{A_i^\top A_i}^2\\
&\|x^k-x^{k+1}\|_{A^\top A}^2=\left\|\sum_{i\in I_k}A_i(x_i^k-x_i^{k+1})\right\|^2\le n\sum_{i=1}^N\|x_i^k-x_i^{k+1}\|_{A_i^\top A_i}^2\\
&\|y^k-y^{k+1}\|_{B^\top B}^2=\left\|\sum_{j\in J_k}B_j(y_j^k-y_j^{k+1})\right\|^2\le m\sum_{j=1}^M\|y_j^k-y_j^{k+1}\|_{B_j^\top B_j}.
\end{align*}
Hence, if $\hat{P}$ and $\hat{Q}$ satisfy \eqref{para-mat-ij}, then \eqref{ineq-k4-w-sub2} also holds.

Combining  \eqref{ineq-k3-w}, \eqref{ineq-k4-w-sub1} and \eqref{ineq-k4-w-sub2} yields
\begin{eqnarray}\label{ineq-k5-w}
& &\EE\left[\Phi(\vx^{t+1},\vy^{t+1})-\Phi(\vx,\vy)+(\tilde{\vw}^{t+1}-\vw)^\top H(\tilde{\vw}^{t+1})\right]\cr
& &+\theta\sum_{k=0}^{t-1}\EE\left[\Phi(\vx^{k+1},\vy^{k+1})-\Phi(\vx,\vy)+(\vw^{k+1}-\vw)^\top H(\vw^{k+1})\right]\cr
& \le & (1-\theta)\left[\Phi(\vx^0,\vy^0)-\Phi(\vx,\vy)\right]\cr
& &+(1-\theta)\left[(\vx^{0}-\vx)^\top(-\vA^\top\vlam^0)+
\rho_x(\vx^{0}-\vx)^\top\vA^\top\vr^{0}+(\vy^{0}-\vy)^\top(-\vB^\top\vlam^0)+
\rho_y(\vy^{0}-\vy)^\top\vB^\top\vr^{0}\right]\cr
& &+ \frac{1}{2}\left(\|\vx^0-\vx\|_{\hat{\vP}-\theta\rho_xA^\top A}^2+\|\vy^0-\vy\|_{\hat{\vQ}}^2\right)+\frac{\theta}{2\rho}\EE\|\lambda^0-\lambda\|^2.
\end{eqnarray}
Applying the convexity of $\Phi$ and the properties \eqref{prop-mas-H} of $H$, we have
\begin{eqnarray}\label{ineq-con-sum}
& &(1+\theta t)\EE\left[\Phi(\hat{\vx}^t,\hat{\vy}^t)-\Phi(\vx,\vy)+(\hat{\vw}^{t+1}-\vw)^\top H(\vw)\right]\cr
&\overset{\eqref{equivHw}}= &(1+\theta t)\EE\left[\Phi(\hat{\vx}^t,\hat{\vy}^t)-\Phi(\vx,\vy)+(\hat{\vw}^{t+1}-\vw)^\top H(\hat{\vw}^{t+1})\right]\cr
&\overset{\eqref{prop-mas-H}}\leq&\EE\left[\Phi(\vx^{t+1},\vy^{t+1})-\Phi(\vx,\vy)+(\tilde{\vw}^{t+1}-\vw)^\top H(\tilde{\vw}^{t+1})\right]\cr
& & +\theta\sum_{k=0}^{t-1}\EE\left[\Phi(\vx^{k+1},\vy^{k+1})-\Phi(\vx,\vy)+(\vw^{k+1}-\vw)^\top H(\vw^{k+1})\right].
\end{eqnarray}
Now combining \eqref{ineq-con-sum} and \eqref{ineq-k5-w}, we have
\begin{eqnarray}\label{ineq-opt-w}
& &(1+\theta t)\EE\left[\Phi(\hat{\vx}^t,\hat{\vy}^t)-\Phi(\vx,\vy)+(\hat{\vw}^{t+1}-\vw)^\top H(\vw)\right]\cr
&\le & (1-\theta)\left[\Phi(\vx^0,\vy^0)-\Phi(\vx,\vy)\right]\cr
& &+(1-\theta)\left[(\vx^{0}-\vx)^\top(-\vA^\top\vlam^0)+
\rho_x(\vx^{0}-\vx)^\top\vA^\top\vr^{0}+(\vy^{0}-\vy)^\top(-\vB^\top\vlam^0)+
\rho_y(\vy^{0}-\vy)^\top\vB^\top\vr^{0}\right]\cr
& &+ \frac{1}{2}\left(\|\vx^0-\vx\|_{\hat{\vP}-\theta\rho_xA^\top A}^2+\|\vy^0-\vy\|_{\hat{\vQ}}^2\right)+\frac{\theta}{2\rho}\EE\|\lambda^0-\lambda\|^2.
\end{eqnarray}

By Lemmas \ref{lem:xy-rate} and \ref{equiv-rate} with $\gamma=\max\{1+\|\lambda^*\|, 2\|\lambda^*\|\}$, we have the desired result. 
%

\subsection{Proof of Theorem \ref{thm-s-vx}}
From the nonincreasing monotonicity of $\alpha_k$, one can easily show the following result.
\begin{lemma} Assume $\vlam^{-1}=\vlam^0$. It holds that
\begin{eqnarray}\label{ineq-k3-x-s}
& &\sum_{k=0}^t\frac{(1-\theta)\beta_k}{2}\left[\|\vlam^{k}-\vlam\|^2-\|\vlam^{k-1}-\vlam\|^2+\|\vlam^{k}-\vlam^{k-1}\|^2\right]\notag\\
& &-\sum_{k=0}^{t-1}\frac{\alpha_k\beta_{k+1}}{2\alpha_{k+1}}\left[\|\vlam^{k+1}-\vlam\|^2-\|\vlam^{k}-\vlam\|^2+\|\vlam^{k+1}-\vlam^k\|^2\right]\nonumber\\
&\le&-\sum_{k=0}^{t-1}\frac{\beta_{k+1}}{2}\|\vlam^{k+1}-\vlam^k\|^2
+\sum_{k=1}^t\frac{(1-\theta)\beta_k}{2}\|\vlam^{k}-\vlam^{k-1}\|^2 \nonumber \\
& & -\sum_{k=0}^{t-1}\frac{\alpha_k\beta_{k+1}}{2\alpha_{k+1}}\|\vlam^{k+1}-\vlam\|^2
-\sum_{k=1}^t\frac{(1-\theta)\beta_k}{2}\|\vlam^{k-1}-\vlam\|^2\nonumber\\
& &+\frac{\alpha_0\beta_1}{2\alpha_1}\|\vlam^0-\vlam\|^2+\sum_{k=1}^{t-1}\frac{\alpha_k\beta_{k+1}}{2\alpha_{k+1}}\|\vlam^{k}-\vlam\|^2
+\sum_{k=1}^t\frac{(1-\theta)\beta_k}{2}\|\vlam^{k}-\vlam\|^2\nonumber\\
&=&-\sum_{k=0}^{t-1}\frac{\theta\beta_{k+1}}{2}\|\vlam^{k+1}-\vlam^k\|^2
+\left(\frac{\alpha_0\beta_1}{2\alpha_1}-\frac{(1-\theta)\beta_1}{2}\right)\|\vlam^0-\vlam\|^2
-\left(\frac{\alpha_{t-1}\beta_t}{2\alpha_t}-\frac{(1-\theta)\beta_t}{2}\right)\|\vlam^{t}-\vlam\|^2\nonumber\\
& &-\sum_{k=1}^{t-1}\left(\frac{\alpha_{k-1}\beta_k}{2\alpha_k}
+\frac{(1-\theta)\beta_{k+1}}{2}-\frac{\alpha_k\beta_{k+1}}{2\alpha_{k+1}}-\frac{(1-\theta)\beta_k}{2}\right)\|\vlam^{k}-\vlam\|^2.
\end{eqnarray}
\end{lemma}

By the update formula of $\vlam$ in \eqref{eq:supdate-lam}, we have from \eqref{ineq-k0-x-s} that
\begin{eqnarray}\label{ineq-k1-x-s}
& &\EE\left[F(\vx^{k+1})-F(\vx)+(\vx^{k+1}-\vx)^\top(-\vA^\top\vlam^{k+1})+(\vlam^{k+1}-\vlam)^\top \vr^{k+1}\right]\cr
& &+\EE\left[\frac{(\vlam^{k+1}-\vlam)^\top(\vlam^{k+1}-\vlam^{k})}{\left(1-\frac{(1-\theta)\alpha_{k+1}}{\alpha_k }\right)\rho}
+\frac{(1-\theta)\alpha_{k+1}}{\alpha_k}\rho(\vx^{k+1}-\vx)^\top\vA^\top\vr^{k+1}\right]\cr
& &+\EE(\vx^{k+1}-\vx)^\top\left(\tilde{\vP}+\frac{\vI}{\alpha_k}\right)(\vx^{k+1}-\vx^k)
-\frac{L_f}{2}\EE\|\vx^k-\vx^{k+1}\|^2+\EE(\vx^{k+1}-\vx^k)^\top\vdelta^k\cr
&\le & (1-\theta)\EE\left[F(\vx^k)-F(\vx)+(\vx^{k}-\vx)^\top(-\vA^\top\vlam^k)+(\vlam^{k}-\vlam)^\top \vr^{k}+\frac{(\vlam^{k}-\vlam)^\top(\vlam^k-\vlam^{k-1})}{\left(1-\frac{(1-\theta)\alpha_{k}}{\alpha_{k-1}}\right)\rho}
\right]\cr
& &+(1-\theta)\rho\EE(\vx^{k}-\vx)^\top\vA^\top\vr^{k},
\end{eqnarray}
where similar to  \eqref{matPQtilde}, we have defined $\tilde{P}=\hat{P}-\rho A^\top A$.

Multiplying $\alpha_k$ to both sides of \eqref{ineq-k1-x-s} and using \eqref{w-H-s} and \eqref{uv-cross}, we have
\begin{eqnarray}\label{ineq-k2-x-s}
& &\alpha_k\EE\left[F(\vx^{k+1})-F(\vx)+(\vw^{k+1}-\vw)^\top H(\vw^{k+1})\right]\cr
& &+\frac{\alpha_k\beta_{k+1}}{2\alpha_{k+1}}\EE\left[\|\vlam^{k+1}-\vlam\|^2-\|\vlam^{k}-\vlam\|^2+\|\vlam^{k+1}-\vlam^k\|^2\right]
+\EE\left[(1-\theta)\alpha_{k+1}\rho(\vx^{k+1}-\vx)^\top\vA^\top\vr^{k+1}\right]\cr
& &+\frac{\alpha_k}{2}\EE\big[\|\vx^{k+1}-\vx\|_{\tilde{\vP}}^2-\|\vx^{k}-\vx\|_{\tilde{\vP}}^2+\|\vx^{k+1}-\vx^k\|_{\tilde{\vP}}^2\big]
+\frac{1}{2}\EE\left[\|\vx^{k+1}-\vx\|^2-\|\vx^{k}-\vx\|^2+\|\vx^{k+1}-\vx^k\|^2\right]\cr
& &-\frac{\alpha_kL_f}{2}\EE\|\vx^k-\vx^{k+1}\|^2+\alpha_k\EE(\vx^{k+1}-\vx^k)^\top\vdelta^k\cr
&\le & (1-\theta)\alpha_k\EE\left[F(\vx^k)-F(\vx)+(\vw^{k}-\vw)^\top H(\vw^k)\right]\cr
& &+\frac{(1-\theta)\beta_k}{2}\EE\left[\|\vlam^{k}-\vlam\|^2-\|\vlam^{k-1}-\vlam\|^2+\|\vlam^{k}-\vlam^{k-1}\|^2\right]+\alpha_k(1-\theta)\rho\EE(\vx^{k}-\vx)^\top\vA^\top\vr^{k}.
\end{eqnarray}

Denote  $\tilde{\vlam}^{t+1}=\vlam^t-\rho\vr^{t+1}.$
Then for $k=t$, it is easy to see that \eqref{ineq-k2-x-s} becomes
\begin{eqnarray}\label{ineq-k2-x-s-t}
& &\alpha_t\EE\left[F(\vx^{t+1})-F(\vx)+(\tilde{\vw}^{t+1}-\vw)^\top H(\tilde{\vw}^{t+1})\right]\cr
& &+\frac{\alpha_t}{2\rho}\EE\left[\|\tilde{\vlam}^{t+1}-\vlam\|^2-\|\vlam^{t}-\vlam\|^2+\|\tilde{\vlam}^{t+1}-\vlam^t\|^2\right]\cr
& &+\frac{\alpha_t}{2}\EE\left[\|\vx^{t+1}-\vx\|_{\tilde{\vP}}^2-\|\vx^{t}-\vx\|_{\tilde{\vP}}^2+\|\vx^{t+1}-\vx^t\|_{\tilde{\vP}}^2\right]
+\frac{1}{2}\EE\left[\|\vx^{t+1}-\vx\|^2-\|\vx^{t}-\vx\|^2+\|\vx^{t+1}-\vx^t\|^2\right]\cr
& &-\frac{\alpha_tL_f}{2}\EE\|\vx^t-\vx^{t+1}\|^2+\alpha_t\EE(\vx^{t+1}-\vx^t)^\top\vdelta^t\cr
&\le & (1-\theta)\alpha_t\EE\left[F(\vx^t)-F(\vx)+(\vw^{t}-\vw)^\top H(\vw^t)\right]\cr
& &+\frac{(1-\theta)\beta_t}{2}\EE\left[\|\vlam^{t}-\vlam\|^2-\|\vlam^{t-1}-\vlam\|^2+\|\vlam^{t}-\vlam^{t-1}\|^2\right]+\alpha_t(1-\theta)\EE\rho(\vx^{t}-\vx)^\top\vA^\top\vr^{t}.
\end{eqnarray}

By the nonincreasing monotonicity of $\alpha_k$, summing  \eqref{ineq-k2-x-s} from $k=0$ through $t-1$ and \eqref{ineq-k2-x-s-t} and plugging \eqref{ineq-k3-x-s} gives
\begin{eqnarray}\label{ineq-k4-x-s}
& &\alpha_t\EE\left[F(\vx^{t+1})-F(\vx)+(\tilde{\vw}^{t+1}-\vw)^\top H(\tilde{\vw}^{t+1})\right]+\theta\alpha_{k+1}\sum_{k=0}^{t-1}\EE\left[F(\vx^{k+1})-F(\vx)+(\vw^{k+1}-\vw)^\top H(\vw^{k+1})\right]\nonumber\\
& &+\frac{\alpha_t}{2\rho}\EE\left[\|\tilde{\vlam}^{t+1}-\vlam\|^2-\|\vlam^{t}-\vlam\|^2+\|\tilde{\vlam}^{t+1}-\vlam^t\|^2\right]\nonumber\\
& &+\frac{\alpha_{t+1}}{2}\EE\|\vx^{t+1}-\vx\|_{\tilde{\vP}}^2
+\sum_{k=0}^t\frac{\alpha_k}{2}\EE\|\vx^{k+1}-\vx^k\|_{\tilde{\vP}}^2+\frac{1}{2}\EE\big[\|\vx^{t+1}-\vx\|^2-\|\vx^{0}-\vx\|^2+\sum_{k=0}^t\|\vx^{k+1}-\vx^k\|^2\big]\nonumber\\
& &-\sum_{k=0}^t\frac{\alpha_kL_f}{2}\EE\|\vx^k-\vx^{k+1}\|^2+\sum_{k=0}^t\alpha_k\EE(\vx^{k+1}-\vx^k)^\top\vdelta^k\nonumber\\
& \le & (1-\theta)\alpha_0\EE\left[F(\vx^0)-F(\vx)+(\vw^{0}-\vw)^\top H(\vw^0)\right]+\alpha_0(1-\theta)\rho(\vx^{0}-\vx)^\top\vA^\top\vr^{0}+\frac{\alpha_0}{2}\|\vx^{0}-\vx\|_{\tilde{\vP}}^2\nonumber\\
& &-\sum_{k=0}^{t-1}\frac{\theta\beta_{k+1}}{2}\EE\|\vlam^{k+1}-\vlam^k\|^2
+\left(\frac{\alpha_0\beta_1}{2\alpha_1}-\frac{(1-\theta)\beta_1}{2}\right)\EE\|\vlam^0-\vlam\|^2
-\left(\frac{\alpha_{t-1}\beta_t}{2\alpha_t}-\frac{(1-\theta)\beta_t}{2}\right)\EE\|\vlam^{t}-\vlam\|^2\nonumber\\
& &-\sum_{k=1}^{t-1}\left(\frac{\alpha_{k-1}\beta_k}{2\alpha_k}+\frac{(1-\theta)\beta_{k+1}}{2}
-\frac{\alpha_k\beta_{k+1}}{2\alpha_{k+1}}-\frac{(1-\theta)\beta_k}{2}\right)\EE\|\vlam^{k}-\vlam\|^2 .
\end{eqnarray}
From \eqref{bd-x-s2}, we have
$$
\frac{\alpha_t}{2\rho}\left[\|\tilde{\vlam}^{t+1}-\vlam\|^2-\|\vlam^{t}-\vlam\|^2
+\|\tilde{\vlam}^{t+1}-\vlam^t\|^2\right]\ge-\left(\frac{\alpha_{t-1}\beta_t}{2\alpha_t}-\frac{(1-\theta)\beta_t}{2}\right)\|\vlam^{t}-\vlam\|^2.
$$
In addition, from Young's inequality, it holds that
$$
\frac{1}{2}\|\vx^{k+1}-\vx^k\|^2+\alpha_k\EE(\vx^{k+1}-\vx^k)^\top\vdelta^k\ge\frac{\alpha^2_k}{2}\|\vdelta\|^2.
$$
Hence, dropping negative terms on the right hand side of \eqref{ineq-k4-x-s}, from the convexity of $\Phi$ and \eqref{prop-mas-H}, we have
\begin{eqnarray}\label{ineq-k4-x-f}
& &\left(\alpha_{t+1}+\theta\sum\limits_{k=1}^t\alpha_k\right)\EE\left[F(\hat{\vx}^{t})-F(\vx)+(\hat{\vw}^{t}-\vw)^\top H(\hat{\vw}^{t})\right]\cr
& &\alpha_t\EE\left[F(\vx^{t+1})-F(\vx)+(\tilde{\vw}^{t+1}-\vw)^\top H(\tilde{\vw}^{t+1})\right]+\theta\alpha_{k+1}\sum_{k=0}^{t-1}\EE\left[F(\vx^{k+1})-F(\vx)+(\vw^{k+1}-\vw)^\top H(\vw^{k+1})\right]\cr
&\le & (1-\theta)\alpha_0\left[F(\vx^0)-F(\vx)+(\vw^{0}-\vw)^\top H(\vw^0)\right]
+(1-\theta)\alpha_0\rho(\vx^{0}-\vx)^\top\vA^\top\vr^{0}+\frac{\alpha_0}{2}\|\vx^{0}-\vx\|_{\tilde{\vP}}^2+\frac{1}{2}\|\vx^0-\vx\|^2\cr
& &+\left(\frac{\alpha_0\beta_1}{2\alpha_1}-\frac{(1-\theta)\beta_1}{2}\right)\EE\|\vlam^0-\vlam\|^2+\sum_{k=0}^t\frac{\alpha_k^2}{2}\EE\|\vdelta^k\|^2.
\end{eqnarray}
Using Lemma \ref{lem:xy-rate} and the properties of $H$, we derive the desired result.

\subsection{Proof of Proposition \ref{prop-equiv-pd}}
Let $(I+\partial\phi)^{-1}(x):=\argmin_z \phi(z)+\frac{1}{2}\|z-x\|_2^2$ denote the proximal mapping of $\phi$ at $x$.
Then the update in \eqref{alg:r1st-pd-z} can be written to
$$z^{k+1}=\left(I+\partial\left(\frac{g^*}{\eta}\right)\right)^{-1}\left(z^k-\frac{1}{\eta}Ax^{k+1}\right).$$
Define $y^{k+1}$ as that in \eqref{equiv-y}. Then
\begin{eqnarray*}
\frac{1}{\eta}y^{k+1}&=&\frac{1}{\eta}\left\{\argmin_y g(y)-\langle y, z^k\rangle+\frac{1}{2\eta}\| y+ A x^{k+1}\|^2\right\}\\
&=&\frac{1}{\eta}\left\{\argmin_y g(y)+\frac{\eta}{2}\|\frac{1}{\eta} y-(z^k-\frac{1}{\eta} A x^{k+1})\|^2\right\}\\
&=&\argmin_y g(\eta y)+\frac{\eta}{2}\|y-(z^k-\frac{1}{\eta} A x^{k+1})\|^2\\
&=&\left(I+\partial\left(\frac{1}{\eta} g(\eta\cdot)\right)\right)^{-1}\left(z^k-\frac{1}{\eta} A x^{k+1}\right).
\end{eqnarray*}
Hence, using the fact that the conjugate of $\frac{1}{\eta}g^*$ is $\frac{1}{\eta} g(\eta\cdot)$ and the Moreau's identity $(I+\partial\phi)^{-1}+(I+\partial\phi^*)^{-1}=I$ for any convex function $\phi$, we have
$$\left(I+\partial\left(\frac{g^*}{\eta}\right)\right)^{-1}\left(z^k-\frac{1}{\eta}Ax^{k+1}\right)+\left(I+\partial\left(\frac{1}{\eta} g(\eta\cdot)\right)\right)^{-1}\left(z^k-\frac{1}{\eta} A x^{k+1}\right)=z^k-\frac{1}{\eta} A x^{k+1}.$$
Therefore, \eqref{equiv-z} holds,
and thus from \eqref{alg:r1st-pd-zbar} it follows
$$\bar{z}^{k+1}=z^{k+1}-\frac{q}{\eta}(Ax^{k+1}+y^{k+1}).$$
Substituting the formula of $\bar{z}^k$ into \eqref{alg:r1st-pd-x}, we have for $i=i_k$,
\begin{align*}
x_i^{k+1}=&\argmin_{x_i\in \cX_i}\langle -\bar{z}^k, A_ix_i\rangle +u_i(x_i)+\frac{\tau}{2}\|x_i-x_i^k\|^2\cr
=& \argmin_{x_i\in \cX_i}\langle -z^k, A_i x_i\rangle+\frac{q}{\eta}\langle A x^k+y^{k}, A_i x_i\rangle+u_i(x_i)+\frac{\tau}{2}\|x_i-x_i^k\|^2\cr
=& \argmin_{x_i\in \cX_i}\langle -z^k, A_i x_i\rangle +u_i(x_i)+\frac{q}{2\eta}\|A_i x_i +A_{\neq i} x_{\neq i}^k+y^k\|^2+\frac{1}{2}\|x_i-x_i^k\|_{\tau I-\frac{q}{\eta}A_i^\top A_i},
\end{align*}
which is exactly \eqref{equiv-x}. Hence, we complete the proof.
}
\end{document}